\documentclass[sn-mathphys,Numbered]{sn-jnl}

\usepackage{mathtools}
\usepackage{amssymb}
\usepackage{graphicx}
\usepackage{multirow}
\usepackage{amsmath,amssymb,amsfonts}
\usepackage{amsthm}
\usepackage{mathrsfs}
\usepackage[title]{appendix}
\usepackage{xcolor}
\usepackage{textcomp}
\usepackage{manyfoot}
\usepackage{booktabs}
\usepackage{algorithm}
\usepackage{algorithmicx}
\usepackage{algpseudocode}
\usepackage{listings}
\usepackage{fontawesome}

\DeclareMathOperator{\tr}{tr}
\DeclareMathOperator{\Adj}{Adj}

\theoremstyle{thmstyleone}
\newtheorem{theorem}{Theorem}
\newtheorem{proposition}{Proposition}
\theoremstyle{thmstyletwo}

\newtheorem{remark}{Remark}
\newtheorem{lemma}{Lemma}
\newtheorem{claim}{Claim}
\theoremstyle{thmstylethree}
\newtheorem{definition}{Definition}
\DeclarePairedDelimiter\abs{\lvert}{\rvert}
\begin{document}

\title[Article Title]{On twisted Koecher-Maass series and its integral Kernel}

\author[1,2]{\fnm{Fernando} \sur{Herrera}}
\affil[1]{\orgdiv{Facultad de Ciencias}, \orgname{Universidad Arturo Prat}, \orgaddress{\street{Av. Arturo Prat 2120}, \city{Iquique}, \postcode{1100000}, \country{Chile}}}
\affil[2]{\orgdiv{Instituto de Ciencias Exactas y Naturales}, \orgname{Universidad Arturo Prat}, \orgaddress{\street{Playa Brava 3256}, \city{Iquique}, \postcode{1111346}, \country{Chile}}}

\abstract{We establish the integral kernel associated with the Koecher-Maass series of degree three twisted by an Eisenstein series. We prove that such a kernel admits an analytic continuation and determine its functional equations. We find a second representation of this kernel using Poincaré series. As an application, we give another proof of the analytic properties for the twisted Koecher-Maass series. Furthermore, we generalize a result due to Siegel related to the Lipschitz summation formula on the Siegel space of degree three.}



\maketitle
\vspace{-1cm}
{\small\faEnvelopeO}~ferhercon@gmail.com
\vspace{0.4cm}
\section{Introduction}

Let $f$ be a Siegel cusp form of even integral weight $k$ over the Siegel modular group $Sp_3(\mathbb{Z})$ and
\vspace{0.1cm}
\begin{equation}\label{fourier}
f(Z)=\sum_{T \in \mathcal{J}} A_T\,e^{2\pi i \tr(TZ)}
\vspace{0.15cm}
\end{equation}
its Fourier series expansion, where $\mathcal{J}:=\mathcal{J}_3$ is the set of half-integral, positive-definite, 3 by 3 matrices, $\{A_T\}_{T \in \mathcal{J}}\subset \mathbb{C}$ the set of its Fourier coefficients and $\tr$ denotes the trace of a square matrix. Let
\vspace{-0.15cm}
\begin{equation}\label{KSdirichlet}
D_f(s)=\sum_{T \in \mathcal{J}/SL_3(\mathbb{Z})} \frac{1}{\varepsilon_T}\frac{A_T}{(\det T)^{s}}
\vspace{-0.1cm}
\end{equation}
be the Koecher-Maass series of degree three associated with $f$, where $s \in \mathbb{C}$, $\Re(s)>2+k/2$, the action of $SL_3(\mathbb{Z})$ over $\mathcal{J}$ is $T[U]:={}^t\hspace{-0.02cm}UTU$ with ${}^t\hspace{-0.02cm}U$ the transpose of $U$ and $\varepsilon_T:=\#\{g \in SL_3(\mathbb{Z})/ T[g]=T\}$. This series is a generalization of the L-function associated with the elliptic cusp forms (i.e. cusp forms over $SL_2({\mathbb{Z}})=Sp_1(\mathbb{Z})$). The series
\vspace{0.2cm}
\begin{equation*}
    D^*_f(s):=(2\pi)^{-3s}\Gamma(s)\Gamma(s-1/2)\Gamma(s-1)D_f(s)
\vspace{0.25cm}
\end{equation*}
has holomorphic continuation to $\mathbb{C}$ and satisfies $D^*_f(k-s)=(-1)^{k/2}D^*_f(s)$, these properties were proven via a generalized Mellin transform; the same analytic properties for the twisted case by Hecke characters, also known as Grossencharacters, were established through invariant differential operators (see \cite{Maa}).

The finite-dimensional complex vector space of weight $k$ Siegel cusp forms over $Sp_3(\mathbb{Z})$, denoted by $\mathfrak{S}_{3,k}$, is equipped with the Petersson inner product. For a fixed complex $s$ with $\Re(s)>2+k/2$, the map $f \mapsto D_f(s)$ is a linear functional of $\mathfrak{S}_{3,k}$; then, there exists $\Omega_{k,s} \in \mathfrak{S}_{3,k}$ such that $D_f(s)=\langle f,\Omega_{k,s} \rangle$ for all $f \in \mathfrak{S}_{3,k}$. The cusp form $\Omega_{k,s}$ is called the integral kernel associated with (\ref{KSdirichlet}).

In \cite{Coh}, Cohen investigates a generating set for elliptic cusp forms, which is composed of the kernels of the periods. Explicit kernels assume a significant role in multiple works, such as \cite{Gol, Gro, Shi, Stu}. The study by Diamantis and O'Sullivan \cite{Dia} focuses on the Cohen kernel for complex values of $s$. In \cite{Koh}, Kohnen and Sengupta study the Koecher Maass series over $Sp_n(\mathbb{Z})$ for any positive integer $n$; the analytic continuation and functional equation of such a series were proven using the related kernel.

In \cite{Mar2}, Martin considers the Koecher-Maass series of Siegel cusp forms of degree two twisted by Epstein zeta function. 
This series can be written in terms of certain Selberg's Eisenstein series $E(Y|s,w)$ in two complex variables $s,w$, namely
\vspace{-0.1cm}
\begin{equation}\label{KMMartin}
D_f(s,w)=\sum_{T \in \mathcal{J}_2/SL_2(\mathbb{Z})}\frac{A_T}{\varepsilon_T} E(T|s,w).
\vspace{-0.1cm}
\end{equation}
In such article Martin studies the integral kernel $\Omega_{k,s,w} \in \mathfrak{S}_{2,k}$, which satisfies (up to an explicit constant $*$) the identity $D_f(s,w)=* \langle f, \Omega_{k,s,w} \rangle$ for all $f \in \mathfrak{S}_{2,k}.$

Our main purpose in this paper is to extend the results in \cite{Mar2} to $\mathfrak{S}_{3,k}$. More precisely, if $\{A_T\}_{T \in \mathcal{J}}$ is the Fourier coefficients of an arbitrary Siegel cusp form $f \in \mathfrak{S}_{3,k}$ as in (\ref{fourier}), its Koecher-Maass series twisted by Selberg's Eisenstein series in three complex variables $E(Y|s,w,u)$ is the following Dirichlet series
\vspace{-0.1cm}
\begin{equation}\label{KMFernando}
D_f(s,w,u)=\sum_{T \in \mathcal{J}/SL_3(\mathbb{Z})}\frac{A_T}{\varepsilon_T} E(T|s,w,u),
\vspace{-0.1cm}
\end{equation}
and for fixed complex numbers $s,w,u$ the map $f \mapsto D_f(s,w,u)$ is a linear functional of $\mathfrak{S}_{3,k}$. In this work we introduce and study a Siegel cusp form ${\Omega}_{k,s,w,u}$, of weight $k$ and over $Sp_3(\mathbb{Z})$, which satisfies (up to an explicit constant $*$) the identity $D_f(s,w,u) =* \langle f, {\Omega}_{k,s,w,u} \rangle$ for all $f \in \mathfrak{S}_{3,k}$.

\subsection{Statement of Results}\label{statements}

Let $\mathcal{P}_n$ be the cone of symmetric, positive-definite matrices in $\mathbb{R}^{n,n}$ and $\mathcal{H}_n$ be the Siegel space of degree $n$, that is $\mathcal{H}_n=\{Z \in \mathbb{C}^{n,n}/\, {}^{t}Z=Z, \Im(Z) \in \mathcal{P}_n\}$. This is a generalization of the complex upper half-plane $\mathcal{H}=\mathcal{H}_1=\{\tau \in \mathbb{C} / \Im(\tau)>0 \}$.

Although some objects are defined for all positive integer $n$, we only consider the cases $n=2,3$ in this work. An important object in this article is the power function
\vspace{-0.05cm}
\begin{equation}\label{potencia}
     \begin{array}{cccl}
 p:&\mathbb{C}^3 \times \mathcal{H}_3 & \longrightarrow & \mathbb{C} \\
 &(s,w,u,Z) & \longmapsto & p_{s,w,u}(Z):=e^{s\,h_1(Z_1)} \,e^{w\,h_2(Z_2)}\,e^{u\,h_3(Z)},
 \end{array}
 \vspace{-0.05cm}
 \end{equation}
where $Z_j \in \mathcal{H}_j$ is the $j\times j$ upper left-hand corner in $Z$ and $h_j$ is certain function which satisfies $e^{h_j(Z_j)}=\det Z_j$ for any $j=1,2,3$ (see \cite[p.~150]{Koh}). The set $\mathcal{S}=\{X \in \mathbb{R}^{3,3}/\, {}^{t}X=X \}$ is a real vector space with inner product $\tr(X\, {}^{t}X)$;  the subsets
\vspace{-0.05cm}
\begin{equation*}
\mathscr{L}=\left\{\left(\begin{smallmatrix}
     a&b/2&c/2\\
     b/2&d&e/2\\
     c/2&e/2&f
     \end{smallmatrix}\right)/ a,b,c,d,e,f \in \mathbb{Z} \right\} \text{ and } \mathscr{L^*}=\left\{ \left(\begin{smallmatrix}
     a&b&c\\
     b&d&e\\
     c&e&f
     \end{smallmatrix}\right) / a,b,c,d,e,f \in \mathbb{Z} \right\}
\vspace{-0.05cm}
\end{equation*}
are lattices in $\mathcal{S}$, dual to each other. The classical Lipschitz summation formula is
\begin{equation}\label{LSForiginal}
    \sum_{n\in \mathbb{Z}} (\tau+n)^{-s}=\frac{(-2\pi i)^s}{\Gamma(s)}\sum_{n\in \mathbb{N}} n^{s-1}e^{2\pi i n\tau },
\vspace{-0.05cm}
\end{equation}
where $\tau \in \mathcal{H}$ and $\Re(s)>1$. Our first proposition generalizes (\ref{LSForiginal}) and is a partial generalization of Siegel’s statement Hilfssatz~38 in \cite[p.~586]{Sie}. Let $W=\left(\begin{smallmatrix}
     0&0&1\\
     0&1&0\\
     1&0&0\\
 \end{smallmatrix}\right).$

\vspace{0.2cm}
\begin{proposition}\label{lipschitz}
Let $(s,w,u,Z) \in \mathbb{C}^3 \times \mathcal{H}_3$ with $\Re(s)>1$, $\Re(w)>3$, $\Re(u)>4.$ Then
\vspace{-0.05cm}
\begin{equation*}
\begin{aligned}
\sum_{B \in \mathscr{L}^*} p_{-s,-w,-u}(Z+B)=-\frac{1}{\pi^{3/2}}&\frac{(-2\pi i)^{s+2w+3u}e^{-(s+2w+3u)\pi i/2}}{\Gamma(s+w+u-1)\Gamma(w+u-1/2)\Gamma(u)}\\[-0.2em] 
&\times \sum_{T\in \mathscr{L}\cap \mathcal{P}_3}p_{-w,-s,s+w+u-2}(iWTW) \, e^{2\pi i \tr(TZ)}.
\end{aligned}
\vspace{-0.05cm}
\end{equation*}
\end{proposition}

The integral kernel associated with the Dirichlet series (\ref{KMMartin}) obtained by Martin is 
\begin{equation*}
\Omega_{k,s,w}(Z)=\sum_{M \in G_{0,2} \backslash Sp_2(\mathbb{Z})} q_{-s,-w}(Z)|_k[M],
\vspace{-0.05cm}
\end{equation*}
where $G_{0,2}$ is certain subgroup of $Sp_2(\mathbb{Z})$, $|_k[\;\;]$ the standard action of $Sp_2(\mathbb{Z})$ in the theory of Siegel modular forms, and
\vspace{-0.1cm}
\begin{equation*}
q_{s,w}(Z)=e^{s\log(\tau_1)}e^{w\log(-\det Z)}
\vspace{-0.1cm}
\end{equation*}
with $Z=\left(\begin{smallmatrix}
     \tau_1&z\\
     z&\tau_2\\
 \end{smallmatrix}\right) \in \mathcal{H}_2$ and $\log$ the principal branch of the logarithm. In this work we establish that the desired generalization, that is the integral kernel associated with the three variables Dirichlet series (\ref{KMFernando}), is
 \vspace{-0.1cm}
\begin{equation}\label{defikernel}
{\Omega}_{k,s,w,u}(Z)=\sum_{M \in G \backslash Sp_3(\mathbb{Z})}p_{-s,-w,-u}(Z)|_k[M]
\vspace{-0.05cm}
\end{equation}
where $G$ is a subgroup of $Sp_3(\mathbb{Z})$ described prior to Definition~\ref{defiintegral}. In Lemma~\ref{AH3} we prove that (\ref{defikernel}) defines a holomorphic function on $A \times \mathcal{H}_3$, where $A=\{(s,w,u) \in \mathbb{C}^3 / \Re(s)>1, \Re(w)>1, \Re(u)>4, \Re{(2s + 4w + u)}<k-4\}$. From now on until the end of this article, the function ${\Omega}_{k,s,w,u}(Z)$ will be called the kernel.

An important family in $\mathfrak{S}_{3,k}$ is the set of Poincaré series $\{P_{k,T} / \ T \in \mathcal{J}\}$ defined in Subsection \ref{cusp}. Our second proposition establishes that the kernel can be written in terms of Poincare series on $B \times \mathcal{H}_3 \subseteq A \times \mathcal{H}_3$, $B=\{(s,w,u) \in \mathbb{C}^3 / \Re(s)>1, \Re(w)>3, \Re(u)>4, \Re{(2s+4w+u)}<k-4\}$, providing us another representation thereof.

\vspace{0.2cm}
\begin{proposition}\label{newrepresentation}
Let $k>22$ and $(s,w,u,Z)\in B \times \mathcal{H}_3$. Then
\vspace{0.03cm}
\begin{equation*}
\begin{aligned}
{\Omega}_{k,s,w,u}(Z)=\frac{2}{\pi^{3/2}}&\frac{(-2\pi i)^{s+2w+3u}}{\Gamma(s+w+u-1)\Gamma(w+u-1/2) \Gamma(u)}\\
& \;\;\;\;\;\;\;\;\;\;\;\;\;\;\;\; \times \sum_{T\in \mathcal{J}/GL_3(\mathbb{Z})} \frac{1}{\varepsilon_T}E(T|w,s,-s-w-u+2)\,P_{k,T}(Z),
\end{aligned}
\vspace{-0.1cm}
\end{equation*}
where $E(Y|s,w,u)$ is Selberg's Eisenstein series defined in Subsection \ref{defSelberg}.
\end{proposition}
\vspace{0.2cm}
Our first theorem lists the analytic properties of the kernel.

\vspace{0.2cm}
\begin{theorem}\label{theoremanalyticcontinuation}
Let $k>22$, $\xi(2s)=\pi^{-s}\Gamma(s)\zeta(2s)$ and $\phi(s)=s(1-s)(s-\frac{1}{2})(\frac{1}{2}-s)$. The series
\vspace{-0.15cm}
\begin{equation}\label{finalfunction}
\phi(s)\xi(2s)\phi(w)\xi(2w)\phi(s+w-1/2)\xi(2s+2w-1){\Omega}_{k,s,w,u}(Z)
\vspace{-0.03cm}
\end{equation}
initially defined on $A \times \mathcal{H}_3$ admits a holomorphic continuation to $\mathbb{C}^3 \times \mathcal{H}_3.$

Such a continuation is a weight $k$ Siegel cusp form over $Sp_3(\mathbb{Z})$ at every triple $(s,w,u) \in \mathbb{C}^3$ and it satisfies the functional equations
\begin{eqnarray*}
e^{(s+2w+3u)\pi i}{\Omega}_{k,s,w,u}(Z) &=& {\Omega}_{k,w,s,-s-w-u+k}(Z) \\
\xi(2w){\Omega}_{k,s,w,u}(Z) &=& \xi(2-2w) {\Omega}_{k,s+w-1/2,1-w,w+u-1/2}(Z) \\
\xi(2s)\xi(2w)\xi(2s+2w-1){\Omega}_{k,s,w,u}(Z) &=& \xi(2-2s)\xi(2-2w)\xi(-2s-2w+3) \\
&&\times\; {\Omega}_{k,1-w,1-s,s+w+u-1}(Z).
\end{eqnarray*}
\end{theorem}

In \cite{Mar2} Martin obtains two functional equations for his integral kernel. Our third proposition shows the structure of the kernel's functional equations.

\vspace{0.2cm}
\begin{proposition}\label{group}
The group generated by the functional equations in Theorem~\ref{theoremanalyticcontinuation} is isomorphic to $D_{12}$, the dihedral group of order twelve.
\end{proposition}
\vspace{0.2cm}

Every $f$ in $\mathfrak{S}_{3,k}$ has associated another $f^* \in \mathfrak{S}_{3,k}$; namely, if (\ref{fourier}) is the Fourier series representation of $f$, then $f^*(Z):=\sum_{T \in \mathcal{J}} \overline{A_T}\;e^{2\pi i \tr(TZ)}.$ If $\langle \;,\; \rangle$ denotes the Petersson inner product (see Subsection \ref{cusp}), our second theorem shows that the considered series in (\ref{defikernel}) is the kernel associated with the Dirichlet series (\ref{KMFernando}).

\vspace{0.2cm}
\begin{theorem}\label{teorema2}
Let $k>22$ and $(s,w,u) \in \mathbb{C}^3$ such that $\Re(s)>1, \ \Re(w)>1$ and $\Re(u)>k+1$. Then
\vspace{-0.25cm}
\begin{equation*}
\langle {\Omega}_{k,s,w,u}, f^* \rangle=\frac{C_k\;(2\pi i)^{-(s+2w+3u)}}{\Gamma(-s-w-u+k)\Gamma(-w-u+k-1/2) \Gamma(-u+k-1)}D_f(s,w,u)
\end{equation*}
for all $f \in \mathfrak{S}_{3,k}$, where $C_k=-2^{12-3k}\pi^6 \, \Gamma(k-2) \, \Gamma(k-5/2) \, \Gamma(k-3)$.
\end{theorem}
\vspace{0.2cm}

As an application of Theorem~\ref{theoremanalyticcontinuation} and Theorem~\ref{teorema2}, we obtain a second proof of the analytic properties of the series (\ref{KMFernando}). The first proof is due to Maass, who used the Hecke character $E(T|s,w,-(s+2w)/3)$, via a generalized Mellin transform and invariant differential operators (see \cite[p.~216]{Maa}).
\vspace{0.2cm}

\begin{proposition} \label{application}
    Let $k>22$ and $f \in \mathfrak{S}_{3,k}$. The multivariable Dirichlet series
\vspace{-0.05cm}
\begin{equation*}
\frac{\phi(s)\xi(2s)\phi(w)\xi(2w)\phi(s+w-1/2)\xi(2s+2w-1)}{\Gamma(-s-w-u+k)\,\Gamma(-w-u+k-1/2)\, \Gamma(-u+k-1)}\,D_f(s,w,u)
\vspace{-0.05cm}
\end{equation*}
initially defined for $\Re(s)>1, \Re(w)>1, \Re(u)>k/2+1$ admits a holomorphic continuation to $\mathbb{C}^3$. Moreover, if we define
\vspace{-0.1cm}
\begin{equation*}
\begin{aligned}
\hspace{0.2cm}    \Lambda_f(s,w,u)=&\,(2\pi)^{-(s+2w+3u)}\Gamma(s+w+u-1)\Gamma(w+u-1/2)\Gamma(u)\\
          &\;\;\;\;\;\;\;\;\;\;\;\;\;\;\;\;\;\;\;\;\;\;\;\;\;\;\;\;\;\;\;\;\;\;\;\;\;\;\;\;\;\;\; \times \xi(2s) \xi(2w)\xi(2s+2w-1)D_f(s,w,u),&
\end{aligned}
\vspace{-0.3cm}
\end{equation*}
then
\vspace{-0.2cm}
\begin{eqnarray*}
\Lambda_f(s,w,u)&=&(-1)^{k/2}\Lambda_f(w,s,-s-w-u+k)\\
&=&\Lambda_f(s+w-1/2,1-w,w+u-1/2)=\Lambda_f(1-w,1-s,s+w+u-1).
\end{eqnarray*}
\end{proposition}

\subsection{Outline of the proofs}

Besides the introduction, the work has four sections. In Section \ref{preliminaries} we recall some basic concepts and prove Proposition~\ref{lipschitz}. In Section \ref{sectionintegralkernel} we study the kernel of the title, and prove Proposition~\ref{newrepresentation} and Theorem~\ref{theoremanalyticcontinuation}; we end the section with Proposition~\ref{group}. In Section \ref{SectionKM} we introduce the twisted Koecher-Maass series and prove Theorem~\ref{teorema2}. In Section \ref{Application} we show an application stated in Proposition~\ref{application}. Now we describe these sections in more detail.

Section $2$: The main result is Proposition~\ref{lipschitz}. Then, we introduce 
a three variables Eisenstein series for $GL_3(\mathbb{Z})$ and recall some particular functions for later use.

Section $3$: We introduce our main object of study, the kernel (\ref{defikernel}). In the beginning of this section we show its first functional equation. Using Proposition~\ref{lipschitz} we prove Proposition~\ref{newrepresentation} in order to get a representation of the kernel as an infinite sum involving Selberg's Eisenstein and Poincaré series.
This new representation, Lemma~\ref{sum} and the first functional equation allow us to establish Proposition~\ref{cuspACWC}, where we obtain that the kernel is a Siegel cusp form for any $(s,w,u)$ on $A\cup C \cup WC \subset \mathbb{C}^3$.
Lemma~\ref{repreonD} is the analytic continuation of $(s,w,u) \mapsto {\Omega}_{k,s,w,u}(Z)$ to $D \subset \mathbb{C}^3$, this is a crucial result which extends the definition of our kernel to certain subsets of $\mathbb{C}^3$ with negative real parts, without loosing the connectivity with the rest of the domain. In Lemma~\ref{Secondfunctionalequationa} we prove a second functional equation for the kernel. In the chain of Lemmas \ref{asets}, \ref{Wsets} and \ref{aWD} we get the analytic continuation of the kernel to some specific regions of $\mathbb{C}^3$, using the two functional equations; the proofs of those lemmas are similar. In Lemma~\ref{thirdeq} we obtain a third functional equation for the kernel. After all this preliminary work, we combine these results and present the proof of Theorem~\ref{theoremanalyticcontinuation}. We end this section showing Proposition~\ref{group}.

Section \ref{SectionKM}: We study the twisted Koecher-Maass series of several variables associated with an arbitrary Siegel cusp form and finish the section with the proof of Theorem~\ref{teorema2}.

Section $5$: We combine Theorem~\ref{theoremanalyticcontinuation} and Theorem~\ref{teorema2} to obtain a new proof of the analytic properties of the twisted Koecher-Maass series stated in Proposition~\ref{application}.

\subsection{Notation}

If $R$ is a commutative ring with identity, we let $R^{n,n}$ be the ring of $n \times n$ matrices with coefficients in $R$. For any $Z \in \mathbb{C}^{n,n}$ we denote by $\tr(Z), \det(Z)$ and ${}^t\hspace{-0.02cm}{Z}$ the trace, determinant and transpose of $Z$ respectively, and define $e(Z)=e^{2\pi i \tr(Z)}$. We write $\Re(Z),\Im(Z) \in \mathbb{R}^{n,n}$ for the real and imaginary part of $Z$. We denote by $I_n$ the identity matrix of size $n$. If $A,B$ are some matrices of appropriate sizes, we put $A[B]={}^t\hspace{-0.02cm}BAB$. Everywhere in this work we let $k$ be an even integer and use the principal branch of the logarithm, i.e. $z^w = e^{w \log (z)}$ for all $w \in \mathbb{C}, z \in \mathbb{C} \setminus  (-\infty,0]$. We assume that the entries of a matrix $T \in \mathcal{J}$, $X \in \mathcal{S},$ $Y \in \mathcal{P}_3$ and $Z \in \mathcal{H}_3$ are
\vspace{-0.1cm}
 \begin{equation}\label{notation}
T=\left(\begin{smallmatrix}
     t_1&t_{12}&t_{13}\\
     t_{12}&t_2&t_{23}\\
     t_{13}&t_{23}&t_3\\
  \end{smallmatrix}\right)\hspace{-0.1cm}, \;
 X=\left(\begin{smallmatrix}
x_1&x_4&x_5\\ 
 x_4&x_2&x_6 \\ 
 x_5&x_6&x_3
\end{smallmatrix}\right)\hspace{-0.1cm}, \;
 Y=\left(\begin{smallmatrix}
y_1&y_4&y_5\\ 
 y_4&y_2&y_6 \\ 
 y_5&y_6&y_3
\end{smallmatrix}\right)\hspace{-0.1cm}, \;
 Z=\left(\begin{smallmatrix}
\tau_1&z_1&z_2\\ 
 z_1&\tau_2&z_3 \\ 
 z_2&z_3&\tau_3
\end{smallmatrix}\right)
\vspace{-0.1cm}
\end{equation}
with $t_1, t_2, t_3, 2t_{12}, 2t_{13}, 2t_{23} \in \mathbb{Z}$, $x_1,...,x_6 \in \mathbb{R}$, $y_1,...,y_6 \in \mathbb{R}$, $\tau_1, \tau_2, \tau_3 \in \mathcal{H}$ and $z_1,z_2,z_3 \in \mathbb{C}$.

\section{Preliminaries}\label{preliminaries}

\subsection{The power and gamma functions of three variables}

In \cite{Koh}, Kohnen and Sengupta prove the existence of a holomorphic function $h_n:\mathcal{H}_n \to \mathbb{C}$, for any positive integer $n$, which satisfies $e^{h_n(Z)}=\det Z$ for all $Z \in \mathcal{H}_n$. These functions are utilized to define our power function $p$ in (\ref{potencia}), which is related to the map studied by Terras in \cite{Ter2}.

The cone $\mathcal{P}_n$ is a $GL_n(\mathbb{R})$-set with the action $Y \mapsto Y[g]={}^t\hspace{-0.02cm}gYg$ for all $Y \in \mathcal{P}_n$, $g \in GL_n(\mathbb{R})$. For the entries $y_j$ of $Y \in \mathcal{P}_3$ as in (\ref{notation}), let $dy_j$ be the Lebesgue measure on $\mathbb{R}$. Then 
\begin{equation*}
d\mu(Y)=(\det Y)^{-2}dy_1 \cdot ... \cdot dy_6
\vspace{0.2cm}
\end{equation*}
is a $GL_3(\mathbb{R})$-invariant volume element on $\mathcal{P}_3$. We define the gamma function for $\mathcal{P}_3$
\vspace{-0.1cm}
\begin{equation*}
\Gamma_3(s,w,u)=\int_{Y \in \mathcal{P}_3}p_{s,w,u}(iY) \, e^{-\tr(Y)} \, d\mu(Y)
\vspace{-0.1cm}
\end{equation*}
where $s,w,u \in \mathbb{C}$. This integral converges if $\Re(s+w+u)>0,$ $\Re(w+u)>1/2$ and $\Re(u)>1$, and satisfies
\begin{equation}\label{Gamma3property}
\Gamma_3(s,w,u)=\pi^{3/2}\,e^{(s+2w+3u)\pi i/2}\,\Gamma(s+w+u)\,\Gamma(w+u-1/2)\,\Gamma(u-1)
\end{equation}
where $\Gamma$ is Euler's gamma function; in particular, $\Gamma_3$ has a meromorphic continuation to $\mathbb{C}^3$. For more details see \cite[p.~20,44,47]{Ter2}.

\vspace{0.2cm}
\begin{lemma}\label{int}
Let $(s,w,u,Z) \in \mathbb{C}^3 \times \mathcal{H}_3$ with $\Re(s+w+u)>0, \Re(w+u)>1/2$ and $\Re(u)>1$. Then
\vspace{-0.15cm}
\begin{equation*}
\int_{Y\in \mathcal{P}_3} p_{s,w,u}(iY)\, e(YZ)\,d\mu(Y)=(2\pi i)^{-s-2w-3u}\; \Gamma_3(s,w,u)\; p_{s,w,u}(-Z^{-1}).
\vspace{-0.05cm}
\end{equation*}
 \end{lemma}
 
 \begin{proof} [Proof]
The proof has two parts: firstly, we split the integral in three pieces and compute them; secondly, we work out the right-hand side of the claimed equality.

For $Z=X+iY$ as in (\ref{notation}), the Cholesky descomposition of $Y$ is 
\vspace{-0.05cm}
\begin{equation}\label{factored}
   Y=\left(\begin{smallmatrix}
     t_1&0&0\\
     t_4&t_2&0\\
     t_5&t_6&t_3
   \end{smallmatrix}\right)
   \left(\begin{smallmatrix}
     t_1&t_4&t_5\\
     0&t_2&t_6\\
     0&0&t_3
   \end{smallmatrix}\right)
   =\left(\begin{smallmatrix}
     t_1^2&t_1t_4&t_1t_5\\
     t_1t_4&t_2^2+t_4^2&t_4t_5+t_2t_6\\
     t_1t_5&t_4t_5+t_2t_6&t_3^2+t_5^2+t_6^2
   \end{smallmatrix}\right).\vspace{-0.1cm}
   \end{equation}
By Lemma~1 of \cite{Koh}, the function $h_n$ satisfies $h_n(iY)=\frac{\pi in}{2}+$ $\log \det Y$ for all $Y \in \mathcal{P}_n$. Using this fact and (\ref{factored}) in (\ref{potencia}) one has 
\vspace{-0.05cm}
\begin{equation*}
 p_{s,w,u}(iY)=e^{(s+2w+3u)\pi i/2} \, t_1^{2(s+w+u)} \, t_2^{2(w+u)} \, t_3^{2u}.
\vspace{-0.1cm}
\end{equation*}
Using the last identity, the expression of $\tr(YZ)$ in terms of $t_j$ from (\ref{factored}) and $d\mu(Y)=8t_1^{-1}t_2^{-2}t_3^{-3}dt_1 \cdot ... \cdot dt_6$ one obtains
\vspace{-0.1cm}
\begin{equation}\label{II}
\begin{aligned}
\frac{e^{-(s+2w+3u)\pi i/2}}{8}&\int_{Y\in \mathcal{P}_3} p_{s,w,u}(iY)e(YZ)d\mu(Y)\\ 
&\hspace{-1cm}=\int_{\substack{t_1,t_2,t_3>0\\t_4,t_5,t_6 \in \mathbb{R}}}t_1^{2(s+w+u)-1}t_2^{2(w+u)-2}t_3^{2u-3}e(YZ)dt_1\cdot ... \cdot dt_6=L_1L_2L_3
\end{aligned}
\vspace{-0.1cm}
\end{equation}
where
\vspace{-0.1cm}
\begin{equation*}
\begin{aligned}
L_1&=\int_{0}^\infty t_3^{2u-3}e^{2\pi i\tau_3t_3^2}dt_3,\\
L_2&=\int_{t_2>0}t_2^{2(w+u)-2}\, e^{2\pi i \tau_2t_2^2}\int_{t_6 \in \mathbb{R}} e^{2\pi i(\tau_3t_6^2+2z_3t_2t_6)}dt_6dt_2, \text{ and }\\
L_3&=\int_{t_1>0}\hspace{-0.29cm}t_1^{2(s+w+u)-1}e^{2\pi i \tau_1t_1^2}\hspace{-0.11cm}\int_{t_4 \in \mathbb{R}}\hspace{-0.29cm}e^{2\pi i(\tau_2t_4^2+2z_1t_1t_4)}\hspace{-0.11cm}\int_{t_5 \in \mathbb{R}}\hspace{-0.29cm}e^{2\pi i (\tau_3t_5^2+2(z_2t_1+z_3t_4)t_5)}dt_5dt_4dt_1.
 \end{aligned}
\end{equation*}
 
Next we work out these integrals separately. Using
\vspace{-0.1cm}
\begin{equation}\label{cor1}
\int_0^\infty t^s e^{2\pi i \tau t^2}dt=\frac{1}{2}\left(-2 \pi i \tau\right)^{-\frac{s+1}{2}}\,\Gamma\left(\frac{s+1}{2}\right)
\end{equation}
for all $\tau \in \mathcal{H}, \Re(s)>-1$ and the identities
\vspace{-0.1cm}
 \begin{equation}\label{identity255}
 (\tau_1 \tau_2^{-1})^s=\tau_1^s \tau_2^{-s}, (-\tau_1)^s=e^{-s\pi i}\tau_1^s, z^s=(z^{-1})^{-s}
\vspace{-0.1cm}
\end{equation}
for all $\tau_1, \tau_2 \in \mathcal{H}, s\in \mathbb{C}$ and $z \in \mathbb{C}-]-\infty,0]$, we can deduce
\vspace{-0.1cm}
\begin{equation}\label{I_1}
 L_1=-\pi i e^{u\pi i} (2\pi i )^{-u} \tau_3^{1-u} \Gamma(u-1).
\vspace{-0.1cm}
\end{equation}
For the inner integral of $L_2$ we use
\vspace{-0.1cm}
\begin{equation}\label{identity4}
\int_{-\infty}^{\infty}e^{-\beta t^2+\gamma t}dt=\sqrt{\frac{\pi}{\beta}}\,e^{\gamma^2/4\beta} \quad \text{ for } \Re(\beta)>0
\vspace{-0.15cm}
\end{equation}
and (\ref{identity255}). Then
\vspace{-0.25cm}
\begin{equation*}
L_2=\frac{\sqrt{\pi} i}{\sqrt{2\pi i}\sqrt{\tau_3}} \int_{t_2>0} t_2^{2(w+u)-2}\,e^{2 \pi i(\tau_2-z_3^2/\tau_3)t_2^2}\,dt_2.
\vspace{-0.1cm}
\end{equation*}
For the computation of this expression we observe that
\vspace{-0.1cm}
\begin{equation*}
\abs{\tau_3}^2\;\;\Im\Bigl(\tau_2-\frac{z_3^2}{\tau_3}\Bigr)=(0,x_3,-x_6)\;Y \; {}^t\hspace{-0.02cm}(0,x_3,-x_6)+y_3 \det \bigl(\Im \left(\begin{smallmatrix}
     \tau_2&z_3\\
     z_3&\tau_3
   \end{smallmatrix}\right)\bigr)>0,
\vspace{-0.1cm}
\end{equation*}   
hence $\tau_2-\frac{z_3^2}{\tau_3} \in \mathcal{H}.$ Furthermore $\frac{\tau_3}{z_3^2-\tau_2\tau_3}=-\bigl(\tau_2-\frac{z_3^2}{\tau_3}\bigr)^{-1}\in \mathcal{H}$ and (\ref{identity255}) imply
\vspace{-0.1cm}
\begin{equation}\label{otheridentity}
\left(\frac{z_3^2-\tau_2\tau_3}{\tau_3}\right)^s=\frac{(z_3^2-\tau_2\tau_3)^s}{\tau_3^s} \ \text{ for all } s\in\mathbb{C}.
\vspace{-0.1cm}
\end{equation}   
Finally $\tau_2-\frac{z_3^2}{\tau_3} \in \mathcal{H}$, (\ref{cor1}) and (\ref{otheridentity}) yield
\vspace{-0.1cm}
\begin{equation}\label{I_2}
L_2=\frac{\sqrt{\pi} i}{2}\; (2\pi i)^{-w-u}\; \frac{(z_3^2-\tau_2\tau_3)^{1/2-w-u}}{\tau_3^{1-w-u}}\; \Gamma(w+u-1/2).
\vspace{-0.1cm}
\end{equation}
For $L_3$, we use (\ref{identity4}) twice in its inner integrals and (\ref{otheridentity}) with $s=1/2$ obtaining
\vspace{-0.1cm}
\begin{equation*}
L_3=\frac{1}{2\sqrt{z_3^2-\tau_2\tau_3}} \int_{t_1>0} t_1^{2(s+w+u)-1}\, e^{2\pi i\bigl(\tau_1-\frac{z_2^2}{\tau_3}-\frac{(\tau_3z_1-z_2z_3)^2}{\tau_3(\tau_2\tau_3-z_3^2)}\bigr)t_1^2}dt_1.
\vspace{-0.1cm}
\end{equation*}
From the inverse of $Z$
\vspace{-0.2cm}
\begin{equation}\label{-z-1}
-Z^{-1}=\frac{1}{\det Z}
\left(\begin{smallmatrix}
     z_3^2-\tau_2\tau_3&z_1\tau_3-z_2z_3&z_2\tau_2-z_1z_3\\
     z_1\tau_3-z_2z_3&z_2^2-\tau_1\tau_3&z_3\tau_1-z_1z_2\\
     z_2\tau_2-z_1z_3&z_3\tau_1-z_1z_2&z_1^2-\tau_1\tau_2
\end{smallmatrix}\right)
\vspace{-0.2cm}
\end{equation}
one sees
\vspace{-0.15cm}
\begin{equation*}
\tau_1-\frac{z_2^2}{\tau_3}-\frac{(\tau_3z_1-z_2z_3)^2}{\tau_3(\tau_2\tau_3-z_3^2)}=\frac{\det Z}{\tau_2\tau_3-z_3^2} ={-((-Z^{-1})_1)^{-1}} \in \mathcal{H}.
\end{equation*}
This last fact allow us to use (\ref{cor1}), which together with (\ref{identity255}) yield
\vspace{-0.1cm}
\begin{equation}\label{I_3}
L_3=\frac{1}{4} \frac{(2\pi i)^{-s-w-u}}{\sqrt{z_3^2-\tau_2\tau_3}} \;  \left(\frac{\det Z}{z_3^2-\tau_2\tau_3}\right)^{-s-w-u} \Gamma(s+w+u).
\vspace{-0.1cm}
\end{equation}

Next, we focus on the right-hand side of the equality of the Lemma. Since $h_n$ is the unique holomorphic function on $\mathcal{H}_n$ satisfying the properties in Lemma~1 of \cite{Koh}, we get\newpage
\vspace{-0.2cm}
\begin{equation}\label{hs}
 \begin{aligned}
 h_1(Z_1)&=\, \log \tau_1,\\
 h_2(Z_2)&=\, \log(-\det Z_2)+\pi i, \text{ and }\\
 h_3(Z)&=\, \log \frac{\det Z}{z_3^2-\tau_2\tau_3}+\log (z_3^2-\tau_2\tau_3)+2\pi i=3\pi i-h_3(-Z^{-1}).
 \end{aligned}
\vspace{-0.1cm}
 \end{equation}
\vspace{-0.1cm}
By (\ref{-z-1}) and $\det ({-Z^{-1}}_2)=\dfrac{\tau_3}{\det Z}$, the expressions in (\ref{hs}) for $-Z^{-1}$ are
\vspace{-0.1cm}
\begin{equation}\label{pinverse}
 \begin{aligned}
 h_1({-Z^{-1}}_1)&=\log \dfrac{z_3^2-\tau_2\tau_3}{\det Z}, \\
 h_2({-Z^{-1}}_2)&=\log \frac{-\tau_3}{\det Z}+\pi i, \text{ and }\\
 h_3(-Z^{-1})&= -\log \frac{\det Z}{z_3^2-\tau_2\tau_3}-\log (z_3^2-\tau_2\tau_3)+\pi i.
\end{aligned}
\vspace{-0.2cm}
 \end{equation}
Using these expressions in (\ref{potencia}) and the identity $e^{w\pi i} \left(\frac{-\tau_3}{\det Z}\right)^{w}\hspace{-0.05cm}=\hspace{-0.05cm}\frac{\tau_3^{w}}{(z_3^2-\tau_2\tau_3)^{w}} \left(\frac{z_3^2-\tau_2\tau_3}{\det Z}\right)^{w}$ (recall $\frac{\det Z}{\tau_2\tau_3-z_3^2}, \frac{\tau_3}{z_3^2-\tau_2\tau_3} \in \mathcal{H}$, (\ref{identity255}),(\ref{otheridentity})) one obtains
\vspace{-0.25cm}
\begin{equation}\label{ps2}
p_{s,w,u}(-Z^{-1})=\frac{e^{u \pi i}\tau_3^{w}}{(z_3^2-\tau_2\tau_3)^{w+u}}\left(\frac{\det Z}{z_3^2-\tau_2\tau_3}\right)^{-s-w-u}.
\vspace{-0.05cm}
\end{equation}
Considering (\ref{Gamma3property}), (\ref{II}), (\ref{I_1}), (\ref{I_2}), (\ref{I_3}) and (\ref{ps2}) we obtain the lemma.
\end{proof}

Consider the vector space $\mathcal{S}$ and the lattices $\mathscr{L},\mathscr{L^*}$, dual to each other, defined in subsection \ref{statements}. For any $Z \in \mathcal{H}_3$, the series $\sum_{B\in \mathscr{L^*}}p_{-s,-w,-u}(Z+B)$ 
is absolutely convergent in the region given by $\Re(s)>1, \, \Re(w)>3$ and $\Re(u)>4$. In order to see this, we use (\ref{hs}) and $Z$ as in (\ref{notation}) obtaining
\vspace{-0.1cm}
\begin{equation*}
\begin{aligned}
 \displaystyle\sum_{B\in \mathscr{L^*}} \abs{p_{-s,-w,-u}&(Z+B)} \\[-0.6em]
 \leq & \abs{e^{-\pi i w}} \displaystyle\sum_{C \in \mathscr{L}^*_2} \abs{ (\tau_1+C_{11})^{-s} } \abs{ (-\det (Z_2+C))^{-w}} \displaystyle\sum_{D \in \mathscr{L}^*}\abs{\det(Z+D)^{-u}}
\end{aligned}
\vspace{-0.15cm}
\end{equation*}
where $\mathscr{L}^*_2=\left\{\left(\begin{smallmatrix}
     a&b\\
     b&c
     \end{smallmatrix}\right)/ \ a,b,c \in \mathbb{Z} \right\}$. The first series on the right-hand side converges with $\Re(s)>1, \Re(w)>3$ (see \cite[p.~9]{Mar2}) and the second one converges for $\Re(u)>4$ (see the series $F_k(Z,u)$ in \cite[p.~152]{Koh}). 
     
Next, we fix the notation
\vspace{0.15cm}
\begin{equation}\label{W}
W=\left(\begin{smallmatrix}
     0&0&1\\
     0&1&0\\
     1&0&0\\
 \end{smallmatrix}\right)
\vspace{0.2cm} \end{equation}
and prove a generalization of a property shown by Terras in \cite[p.~45]{Ter2}.

\vspace{0.2cm}
\begin{claim}\label{claim1}For any $(s_1,s_2,s_3,Z) \in \mathbb{C}^3 \times \mathcal{H}_3$, the power function satisfies
\vspace{-0.1cm}
\begin{equation*}
p_{s_1,s_2,s_3}(-Z^{-1})=e^{(s_1+2s_2+3s_3)\pi i} \; p_{s_2,s_1,-s_1-s_2-s_3}(Z[W]).
\vspace{-0.1cm}
\end{equation*}
\end{claim}

\begin{proof} [Proof] 
By (\ref{hs}) and (\ref{pinverse}) one initially has
\vspace{-0.25cm}
\begin{equation*}
p_{s_1,s_2,s_3}(-Z^{-1})=e^{(s_2+3s_3)\pi i} \, e^{s_1\log(\frac{z_3^2-\tau_2\tau_3}{\det Z})} \, e^{s_2 \log(\frac{-\tau_3}{\det Z})} \, e^{-s_3h_3(Z)};
\vspace{-0.1cm}
\end{equation*}
and from (\ref{hs}) and $h_3(Z)=\log(-\frac{\det Z}{\tau_3})+\log \tau_3+\pi i$ (which is obtained in the same way as the identities in (\ref{hs})) one concludes
\vspace{-0.1cm}
\begin{equation*}\label{p1}
p_{s_1,s_2,s_3}(-Z^{-1})=e^{(2s_1+2s_2+3s_3)\pi i} \, e^{s_1\log(z_3^2-\tau_2\tau_3)} \, e^{s_2\log \tau_3} \, e^{-(s_1+s_2+s_3)h_3(Z)}.
\vspace{-0.1cm}
\end{equation*}
On the other hand, the right-hand side of Claim~\ref{claim1} is computed using (\ref{potencia}) for $Z[W]$, (\ref{hs}) and $h_3(Z[W])=h_3(Z)$ (see \cite[p.~151]{Koh}). Hence, the claim is proven.
\end{proof}

\begin{proof} [Proof\nopunct] {\it of Proposition~\ref{lipschitz}.}
Since this proof is similar to the one of Lemma~2 in \cite[p.~9]{Mar2}, except Claim~\ref{claim1}, we omit the details. By Lemma~\ref{int} and Claim~\ref{claim1}, the Fourier transform of $\phi:\mathcal{S} \to \mathbb{C}$ defined as $\phi(X)=p_{-w,-s,s+w+u-2}(iX)e(XZ)$ if $X \in \mathcal{P}_3$ and $\phi(X)=0$ if $X \in \mathcal{S}\setminus \mathcal{P}_3$, is the map
\vspace{-0.1cm}
\begin{equation*}
\widehat{\phi}(Y)=-\frac{\Gamma_3(-w,-s,s+w+u)}{(-2\pi i)^{s+2w+3u}}\;p_{-s,-w,-u}((Z+Y)[W]).
\vspace{-0.1cm}
\end{equation*}Poisson's summation formula applied to $\phi$ and lattice $\mathscr{L}$ yields Proposition~\ref{lipschitz}.
\end{proof}

\subsection{Siegel cusp forms of degree three\nopunct}\label{cusp}

Every matrix $M=\left(\begin{smallmatrix}
     A&B\\
     C&D
    \end{smallmatrix}\right)$
in $Sp_3(\mathbb{Z})$ acts on $\mathcal{H}_3$ via $Z \mapsto MZ=(AZ+B)(CZ+D)^{-1}$. This left action induces the right $Sp_3(\mathbb{Z})$-action
\vspace{-0.15cm}
\begin{equation}\label{j}
f|_k\,[M](Z)=j(M,Z)^{-k}\, f(MZ)
\vspace{-0.1cm}
\end{equation}
on the set of functions $f:$ $\mathcal{H}_3$ $\to \mathbb{C}$, where $j(M,Z)=\det(CZ+D)$.

\vspace{0.2cm}
\begin{definition}\label{defcusp}
A Siegel modular form of degree $3$ and weight $k$ is a holomorphic function $f:\mathcal{H}_3$ $\to \mathbb{C}$ such that $f|_k\,[M]=f$ for all $M \in Sp_3(\mathbb{Z})$. If $f$ has a Fourier series expansion as (\ref{fourier}), we call it a Siegel cusp form.
\end{definition}

The Petersson inner product on the space $\mathfrak{S}_{3,k}$ is
\vspace{-0.1cm}
\begin{equation*}
\langle f, g \rangle=\int_{Sp_3(\mathbb{Z}) \backslash \mathcal{H}_3} f(Z) \overline{g(Z)} (\det Y)^k dV(Z),
\vspace{-0.3cm}
\end{equation*}
where
\vspace{-0.05cm}
\begin{equation*}
dV(Z)=\det(Y)^{-4}dXdY \text{ and } dXdY=\prod\limits_{j=1}^3 d\sigma_jdt_jdx_jdy_j
\vspace{-0.05cm}
\end{equation*}
for $Z=X+iY$ as in (\ref{notation}) with $\tau_j=\sigma_j+it_j, \, z_j=x_j+iy_j$.

Let $\{A_T\}_{T \in \mathcal{J}}$ be the set of Fourier coefficients of any $f \in \mathfrak{S}_{3,k}$ as in (\ref{fourier}). Since $k$ is fixed, the Fourier coefficients satisfy $A_T = A_{T[U]}$ for all $T \in \mathcal{J}, \, U \in GL_3(\mathbb{Z})$. In $\mathfrak{S}_{3,k}$, an example is the Poincaré series $P_{k,T}$ indexed by $T \in \mathcal{J},$ where $k>6$. Namely,
\vspace{-0.08cm}
\begin{equation*}
P_{k,T}(Z)=2^{-1}\sum_{M \in \mathcal{N} \backslash Sp_3(\mathbb{Z})} e(TZ)|_k[M]
\vspace{-0.12cm}
\end{equation*}
where $\mathcal{N}=\big\{ \bigl(\begin{smallmatrix}
  I_3& B \\ 
 0& I_3 
\end{smallmatrix}\bigr) / \; {}^t\hspace{-0.02cm}B=B \in \mathbb{Z}^{3,3} \big\} \subset Sp_3(\mathbb{Z}).$ It satisfies the identity
\vspace{-0.15cm}
\begin{equation}\label{p,f*}
\langle f, P_{k,T} \rangle=\pi^{3/2} \, (4\pi)^{6-3k} \, \Gamma(k-2) \, \Gamma(k-5/2) \, \Gamma(k-3) \, (\det T)^{2-k} \, A_T
\vspace{-0.1cm}
\end{equation}
(\cite[p.~90]{Kli}). Let $\mathcal{F}$ be the standard Siegel $Sp_3(\mathbb{Z})$-fundamental domain in $\mathcal{H}_3$ (\cite[p.~29]{Kli}).

\vspace{0.2cm}
\begin{lemma}\label{poincarebounded}
Let $k>6$ and $K$ a compact subset of $\mathcal{F}$. Then, there exists a positive constant $C_{K,k}$ such that
\vspace{-0.05cm}
\begin{equation*}
\abs[\big]{P_{k,T}(Z)} \leq C_{K,k} \, (\det 2T)^{2-k/2} \text{ for all }T \in \mathcal{J}, Z \in K. 
\vspace{-0.05cm}
\end{equation*}
\end{lemma}The proof of this result is a simple generalization of the analogous statement for Poincaré series of degree two in \cite[p.~11]{Mar2}. In fact, such a proof can be extended to Poincaré series of degree $n$, for any $n$.

\subsection{Selberg’s Eisenstein series\nopunct}\label{defSelberg}

We denote by $P$ to the minimal parabolic subgroup of $GL_3(\mathbb{Z})$ i.e., $P$ is the set of all upper triangular matrices in $GL_3(\mathbb{Z})$. For $Y \in \mathcal{P}_3\text{ and }(s,w,u) \in \mathbb{C}^3$ we define Selberg’s Eisenstein series associated with $P$
\vspace{-0.02cm}
\begin{equation}\label{def eis}
E(Y| s,w,u)=e^{(s+2w+3u)\pi i/2}\sum_{\gamma \in GL_3(\mathbb{Z})/P} p_{-s,-w,-u}(iY[\gamma]).
\vspace{-0.05cm}
\end{equation}
Let $Y$ be an arbitrary element in $\mathcal{P}_3$. For $(s,w,u)\in \mathbb{C}^3$ with $\Re(s)>1$ and $\Re(w)>1$, the series (\ref{def eis}) is absolutely convergent and defines a holomorphic function in such a region. Furthermore, Selberg’s Eisenstein series has a meromorphic continuation to $\mathbb{C}^3$ and satisfies some functional equations; one of them is
\vspace{-0.02cm}
\begin{equation*}
\xi(2w)E(Y| w,s,-s-w-u+2) = \xi(2-2w)E(Y| 1-w,s+w-1/2,-s-w-u+2),
\vspace{-0.02cm}
\end{equation*}
where $\xi(2s)=\pi^{-s}\Gamma(s)\zeta(2s).$
These facts can be seen in \cite{Ter2}, Section $1.5.1$. Clearly $E(Y[U]|s,w,u)=E(Y|s,w,u)$ for all $U \in GL_3(\mathbb{Z})$ and
\vspace{-0.02cm}
\begin{equation}\label{detE}
E(Y| s,w,u)(\det Y)^q=e^{q \pi i/2} E(Y| s,w,u-q) \text{ for any }q \in \mathbb{C}. 
\end{equation}

\subsection{The Eisenstein series, the Epstein zeta function and the real analytic Eisenstein series\nopunct}\label{eisenstein}

For later purposes we recall here some classical functions.

The space $\mathcal{SP}_2$ is the subset of determinant $1$ matrices in $\mathcal{P}_2$, it is a right $SL_2(\mathbb{Z})$-set via $Y \mapsto Y[\gamma]={}^t\hspace{-0.02cm}\gamma Y \gamma$. The set $\mathcal{H}=\mathcal{H}_1$ is a $SL_2(\mathbb{Z})$-space with the standard modular action. If $\tau \in \mathcal{H}$, we write in this subsection $\tau=\sigma+ti$. These $SL_2(\mathbb{Z})$-spaces $\mathcal{H}$ and $\mathcal{SP}_2$ can be identified using the bijection
\vspace{-0.05cm}
\begin{equation}\label{identification}
\tau \mapsto W_{\tau}=\left(\begin{smallmatrix}
     1/t&-\sigma/t\\
     -\sigma/t& (\sigma^2+t^2)/t\\
 \end{smallmatrix}\right), \text{ or its inverse map }Y=\left(\begin{smallmatrix}
     y_1&y\\
     y& y_2\\
 \end{smallmatrix}\right) \mapsto -\frac{y}{y_1}+\frac{1}{y_1}i.
\vspace{-0.05cm}
\end{equation}
Let $\Gamma_{\infty}=\{\bigl(\begin{smallmatrix}
     \pm 1&m\\
     0&\pm 1\\
 \end{smallmatrix}\bigr) / m\in \mathbb{Z}\} \subset SL_2(\mathbb{Z})$. The Eisenstein series $E_s(\tau)$ and the Epstein zeta function $\mathcal{Z}(Y,s)$ are complex-valued maps on $\mathcal{H}$ and $\mathcal{P}_2$ respectively, given by
\vspace{-0.05cm}
\begin{equation*}
E_s(\tau)=\sum_{\gamma \in \Gamma_{\infty} \backslash SL_2(\mathbb{Z})}\Im(\gamma \tau)^s \;\; \text{ and } \;\; \mathcal{Z}(Y,s)=2^{-1}\sum_{v \in \mathbb{Z}^2-0} {Y[v]}^{-s}
\vspace{-0.1cm}
\end{equation*}
for any $s \in \mathbb{C}, \Re(s)>1$. It is not difficult to show $\zeta(2s)E_s(\tau)=\mathcal{Z}(W_{\tau},s)$ (see for instance \cite[p.~260]{Ter1}).

For $s\in \mathbb{C}, \Re(s)>1$ and $\tau \in \mathcal{H}$ we consider the real analytic Eisenstein series
\vspace{-0.1cm}
\begin{equation*}
\zeta_{\mathbb{Z}^2}(s,\tau)=\sum_{(a,c)\in \mathbb{Z}^2-0} \abs{a+c\tau}^{-2s}.
\vspace{-0.2cm}
\end{equation*}
A straightforward computation shows
\vspace{-0.1cm}
\begin{equation*}
2\mathcal{Z}(Y,s)=\left(\dfrac{\Im(\tau_Y)}{(\det Y)^{1/2}}\right)^s \zeta_{\mathbb{Z}^2}(s,\tau_Y)
\vspace{-0.15cm}
\end{equation*}
with $\tau_Y=\frac{y}{y_1}+\frac{(\det Y)^{1/2}}{y_1}i \ $ for any $Y=\left(\begin{smallmatrix}
     y_1&y\\
     y& y_2\\
 \end{smallmatrix}\right) \in \mathcal{P}_2$ (see for instance \cite[p.~20]{Mar2}). Now, we introduce the function $\zeta^*_{\mathbb{Z}^2}(s,\tau)$ by writing
\vspace{-0.1cm}
\begin{equation*}
\zeta_{\mathbb{Z}^2}(s,\tau)=2\zeta(2s)+\frac{2\pi^{1/2}\Gamma(s-1/2)\zeta(2s-1)}{\Gamma(s)} t^{1-2s}+2\zeta^*_{\mathbb{Z}^2}(s,\tau).
\vspace{-0.1cm}
\end{equation*}
For any fixed $\tau$, the right-hand side of this expression admits a meromorphic continuation in the variable $s$ to the complex plane with a unique singularity, a simple pole at $s=1$. The analytic continuation of the map $\zeta^*_{\mathbb{Z}^2}$ is entire. In the literature  (see \cite[p.~258,259]{Kat}) one finds the estimate
\vspace{-0.1cm}
\begin{equation}\label{zeta*}
\zeta^*_{\mathbb{Z}^2}(s,\tau)=\mathcal{O}\left((\abs{\Im(s)}+1)^{\max\{1,2-2\Re(s)\}+\varepsilon}\;t^{-\max\{1,2\Re(s)\}-\varepsilon}\right)
\vspace{-0.1cm}
\end{equation}
for all $s\in \mathbb{C}$ and $\varepsilon>0$, whenever $t\geq t_0>0$ with $t_0$ an arbitrary positive constant.

\section{Integral kernel}\label{sectionintegralkernel}
We start this section defining the main object of study in our work, the kernel ${\Omega}_{k,s,w,u}(Z)$. Let $G$ be the group
\vspace{-0.15cm}
\begin{equation*}
G=\left\{\left(\begin{smallmatrix}
     S&0\\
     0&{}^t\hspace{-0.02cm}{S}^{-1}\\
 \end{smallmatrix}\right) / \ S=\left(\begin{smallmatrix}
     \pm1&0&0\\
     n&\pm1&0\\
     m&p&\pm1
 \end{smallmatrix}\right) \in GL_3(\mathbb{Z}) \right\}  \subset Sp_3(\mathbb{Z}).
 \vspace{-0.1cm}
\end{equation*}
\begin{definition}\label{defiintegral}
For $s,w,u$ in $\mathbb{C}$ and $Z$ in $\mathcal{H}_3$, let
\vspace{-0.1cm}
\begin{equation}\label{omega}
{\Omega}_{k,s,w,u}(Z)=\sum_{M \in G \backslash Sp_3(\mathbb{Z})}p_{-s,-w,-u}(Z)|_k[M].
\vspace{-0.1cm}
\end{equation}
\end{definition}

Before establishing a convergence region for (\ref{omega}), we introduce a claim.

\vspace{0.2cm}
\begin{claim}\label{Claim2}
   Let $Z=X+iY \in \mathcal{H}_3$, then $\abs{\det Z_j}\geq \det Y_j$ for $j=1,2$.
\end{claim}
\begin{proof}
    The case $j=1$ is clear. For $j=2$ we write $Z_2=X_2+iY_2.$ One knows that for any $Y_2 \in \mathcal{P}_2$ there exists an orthogonal matrix $U \in GL_2(\mathbb{Z})$ such that $Y_2[U]$ is diagonal; and since $\det (Z_2[U])=\det Z_2$ and $\det(Y_2[U])=\det(Y_2)$, we can assume without any loss of generality
\vspace{-0.1cm}
\begin{equation*}
Z_2=X_2+iY_2=\bigl(\begin{smallmatrix}
     \tau_1&x\\
     x&\tau_2\\
 \end{smallmatrix}\bigr)=\bigl(\begin{smallmatrix}
     x_1&x\\
     x&x_2\\
 \end{smallmatrix}\bigr)+i\bigl(\begin{smallmatrix}
     {y}_1&0\\
     0&y_2\\
 \end{smallmatrix}\bigr).
\vspace{-0.1cm}
\end{equation*}
A straightforward computation shows that the inequality $\abs{ \det Z_2 }=\abs{\tau_1\tau_2-x^2}\geq y_1y_2=\det Y_2$ is equivalent to $(x_1x_2-x^2)^2+x_1^2y_2^2+x_2^2y_1^2+2y_1y_2x^2\geq 0,$ which is clearly true and the claim is proven.
\end{proof}
We use $\mathcal{R}=\mathcal{R}_3$ for the standard $GL_3(\mathbb{Z})$-fundamental domain in $ \mathcal{P}_3$, called Minkowski's reduced domain, whose explicit description is found for example in \cite[p.~12]{Kli}. For $k>14$ we define
\vspace{-0.1cm}
\begin{equation*}
A=\left\{ (s,w,u) \in \mathbb{C}^3 / \Re(s)>1, \Re(w)>1, \Re(u)>4, \Re{(2s + 4w + u)} < k-4 \right\}.
\vspace{-0.1cm}
\end{equation*}
\begin{lemma}\label{convergenceonA}
Let $k>14$. The series (\ref{omega}) is absolutely uniformly convergent on $K\times V(\delta)$ for all $\delta>0$, where $K$ is a compact subset of $A$ and
\vspace{-0.06cm}
\begin{equation*}
V(\delta)=\{Z=X+iY \in \mathcal{H}_3 / \tr(X^2)\leq \delta^{-1}, \, Y \geq \delta I_3\}.
\vspace{-0.02cm}
\end{equation*}
\end{lemma}

\begin{proof} [Proof]
We define $G_{0,3} =\{\bigl(\begin{smallmatrix}
     U&0\\
     0&{}^t\hspace{-0.02cm}{U}^{-1}\\
 \end{smallmatrix}\bigr) / \ U \in GL_3(\mathbb{Z}) \}$, then $G \subset G_{0,3} \subset Sp_3(\mathbb{Z})$ and
\vspace{-0.06cm}
\begin{equation}\label{doublesum}
\begin{aligned}
\sum_{M \in G \backslash Sp_3(\mathbb{Z})}&\abs[\big]{p_{-s,-w,-u}(Z)|_k[M]}\\[-0.8em]
&\;\;\;\;=\sum_{M_0 \in G_{0,3} \backslash Sp_3(\mathbb{Z})} \abs{j(M_0,Z)}^{-k} \sum_{M_1 \in G \backslash G_{0,3}} \abs[\big]{p_{-s,-w,-u}(M_0Z)|_k[M_1]}
\end{aligned}
\vspace{-0.06cm}
\end{equation}
where $j(M_0,Z)$ is in (\ref{j}). We put $Z_0=M_0\,Z=X_0+iY_0$ throughout this proof.

Firstly, we prove
\vspace{-0.05cm}
\begin{equation}\label{aclaimed}
\sum_{M_1 \in G \backslash G_{0,3}} \abs[\big]{p_{-s,-w,-u}(Z_0)|_k[M_1]} \leq C_K \;\abs{(\det Z_0)^{-u}}\;E(Y_0 \mid \Re(s),\Re(w),0)
\vspace{-0.05cm}
\end{equation}
for some $C_K>0$ which only depends on $K$. Since $M_1$ runs through $G \backslash G_{0,3}$ if and only if $U$ runs through $GL_3(\mathbb{Z})/P$ (recall $P$ is defined in Subsection~\ref{defSelberg}), one gets
\vspace{-0.03cm}
\begin{equation}\label{tobound}
\begin{aligned}
\sum_{M_1 \in G \backslash G_{0,3}} \abs[\big]{p_{-s,-w,-u}(Z_0)|_k[M_1]}=&\sum_{U \in GL_3(\mathbb{Z})/ P} \abs[\big]{p_{-s,-w,-u}(Z_0[U])}\\
=&\abs{(\det Z_0)^{-u}} \sum_{U \in GL_3(\mathbb{Z})/ P} \abs[\big]{p_{-s,-w,0}(Z_0[U])}.
\end{aligned}
\vspace{-0.05cm}
\end{equation}
Claim~\ref{Claim2} and the representations of $h_1$ and $h_2$ from (\ref{hs}) yield
\vspace{-0.1cm}
\begin{equation}\label{bounde}
\begin{aligned}
\abs[\big]{e^{-sh_1(Z_1)} }&\leq e^{\pi \abs{\Im(s)}}Y_1^{-\Re(s)}\\
\abs[\big]{e^{-wh_2(Z_2)} }&\leq e^{\pi (\Im(w)+\abs{\Im(w)})}(\det Y_2)^{-\Re(w)}.
\end{aligned}
\vspace{-0.1cm}
\end{equation}
From (\ref{tobound}), (\ref{bounde}), the identity $p_{s,w,u}(iY)=e^{(s+2w+3u)\pi i/2}(\det Y_1)^{s}(\det Y_2)^{w}(\det Y)^{u}$
and the definition of Selberg's Eisenstein series (\ref{def eis}), we conclude (\ref{aclaimed}).

For any $Z\in \mathcal{H}_3, s\in \mathbb{C}$ we know that $\abs{(\det{Z})^s}=\abs{\det{Z}}^{\Re(s)}e^{-\Im(s)\Im(h_3(Z))}$
and $|\Im(h_3(Z))| \leq 9\pi/2$ (see \cite[p.~151,153]{Koh}); then $\abs{(\det Z_0)^{-u}} \leq \tilde{C}_K \abs{\det Z_0}^{-\Re(u)}$ for some constant $\tilde{C}_K>0$. Using (\ref{aclaimed}), the series on the left-hand side of (\ref{doublesum}) is bounded from above by a constant only depending on $K$ times the series
\vspace{-0.02cm}
\begin{equation}\label{456}
\sum_{\bigl(\begin{smallmatrix}
     A&B\\
     C&D\\
 \end{smallmatrix}\bigr)=M_0 \in G_{0,3} \backslash Sp_3(\mathbb{Z})} \hspace{-0.45cm} \abs{\det (AZ+B)}^{-\Re(u)} \abs{\det (CZ+D)}^{\Re(u)-k} E(Y_0 \mid \Re(s),\Re(w),0).
\vspace{-0.05cm}
\end{equation}
The coset in $G_{0,3} \backslash Sp_3(\mathbb{Z})$ which contains the matrix $M_0=(\begin{smallmatrix}
     A&B\\
     C&D\\
 \end{smallmatrix}) \in Sp_3(\mathbb{Z})$ is
\vspace{-0.02cm}
\begin{equation*}
G_{0,3} M_0 = \left\{ \left(\begin{smallmatrix}
     {}^t\hspace{-0.02cm}U&0\\
     0&U^{-1}\\
\end{smallmatrix}\right)M_0=\left(\begin{smallmatrix}
     {}^t\hspace{-0.02cm}UA&{}^t\hspace{-0.02cm}UB\\
     U^{-1}C&U^{-1}D\\
\end{smallmatrix}\right)/ \ U \in GL_3(\mathbb{Z})\right\}.
\vspace{-0.02cm}
\end{equation*}
Since $\abs{\det (AZ+B)}=\abs{\det ({}^t\hspace{-0.02cm}UAZ+{}^t\hspace{-0.02cm}UB)}, \abs{\det (CZ+D)}=\abs{\det (U^{-1}CZ+U^{-1}D)}$, $Y_0[U]= \Im\left(\left(\begin{smallmatrix}
     {}^t\hspace{-0.02cm}U&0\\
     0&U^{-1}\\
\end{smallmatrix}\right)M_{0} Z\right)$ and $E$ is invariant under the action of $GL_3(Z)$ on $\mathcal{P}_3$, we can assume without loss of generality that
 $Y_0 \in \mathcal{R}$ in every term of (\ref{456}).

Secondly, we prove that the series on the left-hand side of (\ref{doublesum}) is bounded by
\vspace{-0.05cm}
\begin{equation}\label{serie2}
C' \sum_{\bigl(\begin{smallmatrix}
     A&B\\
     C&D\\
 \end{smallmatrix}\bigr) \in G_{0,3} \backslash Sp_3(\mathbb{Z})} \abs{\det (AZ+B)}^{-\Re(u)} \abs{\det (CZ+D)}^{\Re(2s+4w+u)-k}
\vspace{-0.05cm}
\end{equation}
for some constant $C'=C'(\delta,K)>0$. There exists $C>0$ such that
\vspace{-0.05cm}
\begin{equation}\label{E<p}
E(Y_0 \mid \Re(s),\Re(w),0) \leq C E(I_3 \mid \Re(s),\Re(w),0) {{Y_0}_1}^{-\Re(s)}(\det {Y_0}_2)^{-\Re(w)}
\vspace{-0.05cm}
\end{equation}
for all $Y_0 \in \mathcal{R}$ and $\Re(s)>1, \Re(w)>1$ (see \cite[p.~51]{Mar3}); here note that ${Y_0}_{1} ({Y_0}_2)$ is the $1$ by $1$ ($2$ by $2$) upper left-hand corner in $Y_0$. We know $E(I_3 \mid \Re(s),\Re(w),0)$ is bounded by a constant which only depends on $K$ for all $(s,w,u) \in K$ (see Subsection \ref{defSelberg}).

Now we get an upper bound for ${{Y_0}_1}^{-\Re(s)}(\det {Y_0}_2)^{-\Re(w)}$.  If we denote by $y_j$ the
$(j,j)$-th entry of $Y_0$, and use that $Y_0 \in \mathcal{R}$ (and therefore ${Y_0}_2
\in  \mathcal{R}_2$) we have $y_1^2 \leq y_1y_2 \leq
c (\det{{Y_0}_2})$ for some constant $c > 0$ (see for instance \cite[p.~12,13]{Kli}). Thus
\vspace{-0.1cm}
\begin{equation*}
 (\det{{Y_0}_2})^{-\Re{(w)}} \leq c^{\Re{(w)}} y_1^{-2\Re{(w)}} \leq c_K  {{Y_0}_1}^{-2\Re{(w)}},
\vspace{-0.1cm}
\end{equation*}
using that $w$ runs on a compact subset of $\mathbb{C}$. Now, we recall the following statement proved in \cite[p.~63]{Kli}: There exists a constant $\alpha = \alpha(\delta) > 0$ such that
\vspace{-0.05cm}
\begin{equation}\label{f eq 101}
{Y_0}_1 \geq \alpha \; \Im{(M_0 iI_3)_1} \left| j( M_0, iI_3)\right|^2 \left| j(M_0, Z)\right|^{-2}
\vspace{-0.05cm}
\end{equation}
for all $Z \in V(\delta), M_0 \in Sp_3(\mathbb{Z})$. It is well-known
that $M_0 = \left(\begin{smallmatrix} * & * \\
                  C & D\end{smallmatrix}\right)$ implies
\vspace{-0.1cm}
\begin{equation*}
\begin{aligned}
\hfilneg \Im{(M_0 iI_3)} = {}^t\hspace{-0.1cm}\left(D+Ci\right)^{-1} \Im{(iI_3)} \left(D-Ci\right)^{-1}&=\left( \left(D - iC\right) \; {}^t\hspace{-0.1cm} \left(D + iC\right)\right)^{-1} \hspace{10000pt minus 1fil} \\
& \;\;\;\;= \left|\det{(D + iC)}\right|^{-2} \Adj{(C\,{}^t\hspace{-0.02cm}C+D\,{}^t\hspace{-0.02cm}D}),
\end{aligned}
\vspace{-0.1cm}
\end{equation*}
where $\Adj$ denotes the adjoint matrix (see \cite[p.~3]{Kli}). In particular
\vspace{-0.1cm}
\begin{equation*}
\Im{(M_0 iI_3)}_1 = \left|j(M_0, iI_3)\right|^{-2} \Adj{(C\,{}^t\hspace{-0.02cm}C+D\,{}^t\hspace{-0.02cm}D})_1,
\vspace{-0.12cm}
\end{equation*}
and using it in (\ref{f eq 101}) one obtains ${{Y_0}_1} \geq \alpha \Adj{(C\,{}^t\hspace{-0.02cm}C+D\,{}^t\hspace{-0.02cm}D})_1  \left| j( M_0, Z)\right|^{-2}$. Next, since $\Adj{(C\,{}^t\hspace{-0.02cm}C+D\,{}^t\hspace{-0.02cm}D})_1=\det (((C\,{}^t\hspace{-0.02cm}C+D\,{}^t\hspace{-0.02cm}D)[W])_2)$
and $C\,{}^t\hspace{-0.02cm}C+D\,{}^t\hspace{-0.02cm}D \in \mathcal{P}_3$, we deduce that $\Adj{(C\,{}^t\hspace{-0.02cm}C + D\,{}^t\hspace{-0.02cm}D)}_1$ is a positive integer ($W$ is the matrix in (\ref{W})). Furthermore, since $\Re{(s)} + 2\Re{(w)} > 0$ we can conclude
\vspace{-0.1cm}
\begin{equation*}
{{Y_0}_1}^{-\Re(s)}(\det {Y_0}_2)^{-\Re(w)} \leq c_K {{Y_0}_1}^{-\Re(s)-2\Re(w)} \leq \alpha' 
\left| j( M_0, Z)\right|^{2\Re{(s+2w)}}
\vspace{-0.05cm}
\end{equation*}
for some constant $\alpha'=\alpha'(\delta,K) > 0$, obtaining the bound.

From (\ref{456}), using the previous remarks in (\ref{E<p}), one obtains that the left-hand side of (\ref{doublesum}) is bounded by (\ref{serie2}).

If $\bigl(\begin{smallmatrix}
     A&B\\
     C&D\\
 \end{smallmatrix}\bigr)$ runs through a complete set of representatives for $G_{0,3} \backslash Sp_3(\mathbb{Z})$,
both $(C,D)$ and $(A,B)$ run over the collection of non-left-associated coprime symmetric matrix pairs. Consequently, the series in (\ref{serie2}) is bounded by
\vspace{-0.05cm}
\begin{equation}\label{serie3}
 \sum_{\{C,D\}} \abs{\det(CZ+D)}^{-\Re(u)} \sum_{\{C,D\}} \abs{\det(CZ+D)}^{\Re{(2s+4w+u)}-k}
\vspace{-0.13cm}
\end{equation}
where each sum in (\ref{serie3}) is over all non-left-associated coprime symmetric pairs. It is well-known that
\vspace{0.1cm}
\begin{equation*}
 \sum_{\{C,D\}} \abs{\det(CZ+D)}^{-r}
 \vspace{0.2cm}
\end{equation*}
is uniformly convergent on $K_2 \times V(\delta)$ where $K_2$ is a compact subset of $\{r\in \mathbb{Z} / r>4\}$ (for the final facts, see \cite[p.~153]{Koh}). Applying this to (\ref{serie3}), we obtain Lemma~\ref{convergenceonA}.
\end{proof}

We remark that $\mathcal{F}$, the standard Siegel $Sp_3(\mathbb{Z})$-fundamental domain in $\mathcal{H}_3$, is contained in $V(\delta)$ for some positive number $\delta$ (see for instance \cite[p.~29,30]{Kli}).

\vspace{0.2cm}
\begin{lemma}\label{AH3}
The series (\ref{omega}) is absolutely convergent and defines a holomorphic function on $A \times \mathcal{H}_3$.
\end{lemma}

\begin{proof} [Proof]
Firstly, we show that (\ref{omega}) is absolutely convergent on $A \times \mathcal{H}_3$. Let $(s,w,u,Z)$ be an arbitrary element in $A\times\mathcal{H}_3$ and pick $M_Z\in Sp_3(\mathbb{Z})$ such that $M_ZZ\in \mathcal{F}\subseteq V(\delta).$ Since $j(MN,Z)=j(M,NZ)j(N,Z)$ for all $M,N \in Sp_3(\mathbb{Z})$, and the map $GM\mapsto GMM_Z$ on $G \backslash Sp_3(\mathbb{Z})$ is bijective, we have
\vspace{-0.05cm}
\begin{equation*}
\begin{aligned}
&\sum_{M \in G \backslash Sp_3(\mathbb{Z})}\abs{j(M,Z)^{-k}p_{-s,-w,-u}(MZ)} \\[-0.1em]
=&\,\abs{j({M_Z}^{-1},M_ZZ)}^k\sum_{M \in G \backslash Sp_3(\mathbb{Z})}\abs{j(M{M_Z}^{-1},M_ZZ)^{-k}p_{-s,-w,-u}(M{M_Z}^{-1}M_ZZ)}\\[-0.1em]
=&\,\abs{j({M_Z}^{-1},M_ZZ)}^k\sum_{M \in G \backslash Sp_3(\mathbb{Z})} \abs[\big]{j(M,M_ZZ)^{-k}p_{-s,-w,-u}(MM_ZZ)}\\[-0.1em]
=&\,\abs{j({M_Z}^{-1},M_ZZ)}^k\sum_{M \in G \backslash Sp_3(\mathbb{Z})} \abs[\big]{p_{-s,-w,-u}(M_ZZ)|_k[M]}.\\[-0.1em]
\end{aligned}
\vspace{-0.06cm}
\end{equation*}
As $(s,w,u,M_ZZ) \in A\times V(\delta)$, the last series converges by Lemma~\ref{convergenceonA} and the absolute convergence of (\ref{omega}) on $A \times \mathcal{H}_3$ is proved. In particular, the series (\ref{omega}) is a well-defined function on $A \times \mathcal{H}_3$.

We note that for any $(s,w,u) \in A$, the series (\ref{omega}) is invariant under the action of $Sp_3(\mathbb{Z})$: by the absolute convergence just proved we see that the previous argument works without the absolute values and for any matrix $N$ instead of $M_Z$. Then
\vspace{-0.05cm}
\begin{equation}\label{abcde}
{\Omega}_{k,s,w,u}(Z)=j(N^{-1},NZ)^k {\Omega}_{k,s,w,u}(NZ)={\Omega}_{k,s,w,u}(Z)|_k[N]
\vspace{-0.02cm}
\end{equation}
for all $(s,w,u,Z)\in A\times \mathcal{H}_3, N \in Sp_3(\mathbb{Z}).$

Secondly, we prove that the series (\ref{omega}) is uniformly convergent on subsets $K \times \widetilde{K}$ for $K, \widetilde{K}$ arbitrary compact subsets of $A$ and $\mathcal{H}_3$ respectively. Let us suppose the existence of a matrix $\widetilde{M} \in Sp_3(\mathbb{Z})$ such that $\widetilde{M}\widetilde{K} \subseteq \mathcal{F}\subseteq V(\delta)$; from the identity ${\Omega}_{k,s,w,u}(Z)=j(\widetilde{M},Z)^{-k}{\Omega}_{k,s,w,u}(\widetilde{M}Z)$ in (\ref{abcde}), and considering Lemma~\ref{convergenceonA}, we conclude ${\Omega}_{k,s,w,u}(Z)$ is uniformly convergent on $K\times \widetilde{K}$. Since each compact subset of $\mathcal{H}_3$ is covered by at most finitely many images of $\mathcal{F}$ under $Sp_3(\mathbb{Z})$ (see \cite[p.~31]{Kli}), the general case follows from the previous assumption.

For any $M \in Sp_3(\mathbb{Z})$, the map
\vspace{-0.1cm}
\begin{equation*}
    (s,w,u,Z) \mapsto p_{-s,-w,-u}(Z)|_k[M]=j(M,Z)^{-k}e^{-sh_1((MZ)_1)}e^{-wh_2((MZ)_2)}e^{-uh_3(MZ)}
\vspace{-0.02cm}
\end{equation*}
is holomorphic on $A\times \mathcal{H}_3$, and using the uniform convergence just proved, we obtain that the series (\ref{omega}) defines a holomorphic function on $A \times \mathcal{H}_3$.
\end{proof}

The region obtained in Lemma~\ref{AH3} admits an improvement. For $k>14$ we define
\vspace{-0.1cm}
\begin{equation*}
WA=\{ (s,w,u)\in \mathbb{C}^3 / \Re(s)>1, \Re(w)>1, \Re(s+w+u)<k-4, \Re(3s+w-u)<-4\}.
\vspace{-0.1cm}
\end{equation*}

\begin{lemma}\label{wequation}
Let $k>18$. The series (\ref{omega}) satisfies
\vspace{-0.02cm}
\begin{equation}\label{Wequation}
e^{(s+2w+3u)\pi i}{\Omega}_{k,s,w,u}(Z)={\Omega}_{k,w,s,-s-w-u+k}(Z)
\vspace{-0.02cm}
\end{equation}
for all $(s,w,u,Z) \in (A \cap WA)\times \mathcal{H}_3.$ Consequently, ${\Omega}_{k,s,w,u}(Z)$ admits a holomorphic continuation from $A \times \mathcal{H}_3$ to $(A\cup WA) \times \mathcal{H}_3$ via (\ref{Wequation}).
\end{lemma}

\begin{proof} [Proof]
Defining $\widetilde{W}=\bigl(\begin{smallmatrix}
    0&-W\\
    W&0\
 \end{smallmatrix}\bigr)$ with $W$ in (\ref{W}), we use $Z[W]$ in Claim~\ref{claim1} obtaining
\vspace{-0.05cm}
\begin{equation*}\label{104b}
p_{s,w,u}(Z)|_k[\widetilde{W}]=(\det Z)^{-k}p_{s,w,u}(-Z^{-1}[W])=e^{(s+2w+3u) \pi i}p_{w,s,-s-w-u-k}(Z).
\vspace{-0.05cm}
\end{equation*}
The invariance in (\ref{abcde}) with $\widetilde{W}$, the absolute convergence in Lemma~\ref{AH3}, the previous identity and the bijectivity of the map $GM \mapsto G\widetilde{W}^{-1}M\widetilde{W}$ on $G \backslash Sp_3(\mathbb{Z})$ give us
\vspace{-0.1cm}
\begin{equation*}
\begin{aligned}
\hfilneg {\Omega}_{k,s,w,u}(Z)&={\Omega}_{k,s,w,u}(Z)|_k[\widetilde{W}] \hspace{10000pt minus 1fil}\\[-0.1em]
&=\sum_{M \in G \backslash Sp_3(\mathbb{Z})}p_{-s,-w,-u}(Z)|_k[M\widetilde{W}]\\[-0.1em]
&=\sum_{M \in G \backslash Sp_3(\mathbb{Z})}\left(p_{-s,-w,-u}(Z)|_k[\widetilde{W}]\right) |_k[\widetilde{W}^{-1}M\widetilde{W}]\\[-0.1em]
\end{aligned}
\vspace{-0.1cm}
\end{equation*}
\begin{equation*}
\begin{aligned}
&=e^{-(s+2w+3u)\pi i} \sum_{M \in G \backslash Sp_3(\mathbb{Z})}p_{-w,-s,s+w+u-k}(Z)|_k[\widetilde{W}^{-1}M\widetilde{W}]\\[-0.1em]
&=e^{-(s+2w+3u)\pi i}{\Omega}_{k,w,s,-s-w-u+k}(Z).\\[-0.1em]
\end{aligned}
\vspace{0.3cm}
\end{equation*}
The argument above proves the equality in (\ref{Wequation}) for $(s,w,u) \in A$ and $(w,s,-s-w-u+k) \in A$. As the last condition is equivalent to $(s,w,u) \in WA$, the functional equation (\ref{Wequation}) holds on $A\cap WA$ (which is not empty, as $k >18$). Defining ${\Omega}_{k,s,w,u}(Z)$ on $WA$ via (\ref{Wequation}) and using Lemma~\ref{AH3}, we conclude that $(s,w,u,Z)\mapsto {\Omega}_{k,s,w,u}(Z)$ is a well-defined, holomorphic function on $(A \cup WA) \times \mathcal{H}_3$.
\end{proof} 

In order to enlarge the domain of the kernel, which is $(A \cup WA) \times \mathcal{H}_3$ so far, we look for another representation of the series (\ref{omega}).
For $k>22$, we define the set $B$ as the intersection between the set of Proposition~\ref{lipschitz} and $A$:
\vspace{-0.1cm}
\begin{equation*}
 B=\{ (s,w,u) \in \mathbb{C}^3 / \Re(s)>1, \Re(w)>3, \Re(u)>4, \Re(2s+4w+u)<k-4\}.
\vspace{-0.1cm}
\end{equation*}

\begin{lemma}\label{first half}
Let $k>22$ and $(s,w,u,Z) \in B \times \mathcal{H}_3$. Then ${\Omega}_{k,s,w,u}(Z)$ is equal to
\vspace{-0.1cm}
\begin{equation*}
\frac{(2\pi)^{s+2w+3u}}{\Gamma_3(-w,-s,s+w+u)}\sum_{\{C,D\}}\Bigl(\sum_{T \in \mathscr{L}\cap \mathcal{P}_3}\hspace{-0.3cm}E(T| w,s,-s-w-u+2)e(TZ)\Bigr)|_k [\left(\begin{smallmatrix}
     A_0&B_0\\
     C&D\\
   \end{smallmatrix}\right)]
\vspace{-0.1cm}
\end{equation*}
where $\{C,D\}$ is as in (\ref{serie3}) and each $A_0, B_0$ is chosen so that $\bigl(\begin{smallmatrix}
     A_0&B_0\\
     C&D\\
   \end{smallmatrix}\bigr) \in Sp_3(\mathbb{Z})$.
 \end{lemma}

 \begin{proof} [Proof]
 The group $G_{0,3}$, defined prior to (\ref{doublesum}), is contained in the normalizer of $\mathcal{N}$ introduced in Subsection \ref{cusp}, thus $G \leq G_{0,3} \leq G_{0,3}\mathcal{N} \leq Sp_3(\mathbb{Z})$ is a chain of groups. It is well-known that the set of matrices in $Sp_3(\mathbb{Z})$ whose lower $3$ by $3$ blocks $(C, D)$ run over all non left-associated coprime symmetric matrix pairs is a complete set of representatives for $G_{0,3}\mathcal{N} \backslash Sp_3(\mathbb{Z})$ (see for instance \cite[p.~157]{Maa}). Since $G_{0,3} \cap \mathcal{N}=\{I_6\}$, a complete set of representatives for $G_{0,3} \backslash G_{0,3}\mathcal{N}$ is $\mathcal{N}$. Furthermore, the set of matrices $\bigl(\begin{smallmatrix}
     U&0\\
     0&^t{U}^{-1}\\
   \end{smallmatrix}\bigr)$ with $U$ running over ${}^t\hspace{-0.02cm}P\backslash GL_3(\mathbb{Z})$ is a complete set of representatives for $G\backslash G_{0,3}$ (as it was mentioned prior to (\ref{tobound})). Thus, by the absolute convergence proved in Lemma~\ref{AH3}, we have the representation
\vspace{-0.1cm}
\begin{equation*}
\begin{aligned}
{\Omega}_{k,s,w,u}(Z)=&\sum_{M \in G \backslash Sp_3(\mathbb Z)}p_{-s,-w,-u}(Z)|_k[M]\\[-0.2em]
=&\sum_{\{C,D\}} \sum_{U \in \, \prescript{t}{}{P}\backslash GL_3(\mathbb{Z})} \sum_{S \in \mathscr{L}^*} p_{-s,-w,-u}(Z)|_k[\bigl(\begin{smallmatrix}
     I_3&S\\
     0&I_3\\
   \end{smallmatrix}\bigr)\bigl(\begin{smallmatrix}
     U&0\\
     0&^t{U}^{-1}\\
   \end{smallmatrix}\bigr)  \bigl(\begin{smallmatrix}
     A_0&B_0\\
     C&D\\
   \end{smallmatrix}\bigr)]\\[-0.3em]
=&\sum_{\{C,D\}} \sum_{U \in \,{^t}\hspace{-0.02cm}P \backslash GL_3(\mathbb{Z})} \biggl(\sum_{S \in \mathscr{L}^*} p_{-s,-w,-u}(Z+S)\biggr)|_k[\bigl(\begin{smallmatrix}
     U&0\\
     0&^t\hspace{-0.02cm}{U}^{-1}\\
   \end{smallmatrix}\bigr)  \bigl(\begin{smallmatrix}
     A_0&B_0\\
     C&D\\
   \end{smallmatrix}\bigr)].\\[-0.1em]
\end{aligned}
\vspace{-0.1cm}
\end{equation*}
Applying Proposition~\ref{lipschitz} to the inner series and the identity $e(TZ)|_k[\bigl(\begin{smallmatrix}
     U&0\\
     0&^t\hspace{-0.02cm}{U}^{-1}\\
\end{smallmatrix}\bigr)]=e(T[U]Z)$ we get
\vspace{-0.1cm}
\begin{equation*}
\begin{aligned}
  &-\frac{\Gamma_3(-w,-s,s+w+u)}{(-2\pi i)^{s+2w+3u}}{\Omega}_{k,s,w,u}(Z)\\[-0em]
=&\sum_{\{C,D\}} \sum_{U \in \, {^t}\hspace{-0.02cm}P\backslash GL_3(\mathbb{Z})} \sum_{T \in \mathscr{L}\cap \mathcal{P}_3}\hspace{-0.1cm}p_{-w,-s,s+w+u-2}(iT[W])e(TZ)|_k[\bigl(\begin{smallmatrix} U&0\\ 0&^t\hspace{-0.02cm}{U}^{-1}\\ \end{smallmatrix}\bigr) \bigl(\begin{smallmatrix} A_0&B_0\\ C&D\\ \end{smallmatrix}\bigr)]\\[-0em]
\end{aligned}
\end{equation*}
\begin{equation}\label{descomposition}
\begin{aligned}
&=\sum_{\{C,D\}} \sum_{U \in \, {^t}\hspace{-0.02cm}P\backslash GL_3(\mathbb{Z})} \sum_{T \in \mathscr{L}\cap \mathcal{P}_3}\hspace{-0.1cm}p_{-w,-s,s+w+u-2}(iT[W])e(T[U]Z)|_k[\bigl(\begin{smallmatrix} A_0&B_0\\ C&D\\ \end{smallmatrix}\bigr)]\\
&=\sum_{\{C,D\}} \sum_{U \in \, {^t}\hspace{-0.02cm}P\backslash GL_3(\mathbb{Z})} \sum_{T \in \mathscr{L}\cap \mathcal{P}_3}\hspace{-0.1cm}p_{-w,-s,s+w+u-2}(iT[U^{-1}][W])e(TZ)|_k[\bigl(\begin{smallmatrix} A_0&B_0\\ C&D\\ \end{smallmatrix}\bigr)].
\end{aligned}
\vspace{-0.1cm}
\end{equation}
The last equality is by the bijective map $T \mapsto T[U]$ onto $\mathscr{L}\cap \mathcal{P}_3$, for any $U \in GL_3(\mathbb{Z})$. Since ${}^t\hspace{-0.02cm}PW=WP$ and the conjugation by $W$ is a bijection onto $GL_3(\mathbb{Z})$, one gets
\vspace{-0.2cm}
\begin{equation*}
\Big\{ U^{-1}W /\; U \in {}^t\hspace{-0.02cm}P \backslash GL_3(\mathbb{Z}) \Big\}=\left( GL_3(\mathbb{Z})/\,{}^t\hspace{-0.02cm}P\right) W=W \, \left(GL_3(\mathbb{Z})/P\right);
\vspace{-0.2cm}
\end{equation*}
and by the invariance of Selberg's Eisenstein series under the action of $W$ we obtain
\vspace{-0.1cm}
\begin{equation*}
\begin{aligned}
& \sum_{T \in \mathscr{L}\cap \mathcal{P}_3} \sum_{U \in {^t}\hspace{-0.02cm}P\backslash GL_3(\mathbb{Z})} p_{-w,-s,s+w+u-2}(iT[U^{-1}][W])e(TZ)\\
=&\sum_{T \in \mathscr{L}\cap \mathcal{P}_3} \; \sum_{U \in \, GL_3(\mathbb{Z})/P} p_{-w,-s,s+w+u-2}(iT[W][U])e(TZ)\\
=&-e^{(s+2w+3u)\pi i/2} \sum_{T \in \mathscr{L}\cap \mathcal{P}_3} E(T[W]| w,s,-s-w-u+2)\, e(TZ)\\
=&-e^{(s+2w+3u)\pi i/2} \sum_{T \in \mathscr{L}\cap \mathcal{P}_3} E(T| w,s,-s-w-u+2) \, e(TZ).
\end{aligned}
\vspace{-0.1cm}
\end{equation*}
Using this computation of the double inner series in (\ref{descomposition}), we get the formula.
\end{proof}

\begin{proof} [Proof\nopunct]{\it of Proposition~\ref{newrepresentation}.} Note $ \mathcal{J} = \mathscr{L}\cap \mathcal{P}_3$. By Lemma~\ref{first half} and the invariance of Selberg's Eisenstein series one can write
\vspace{-0.1cm}
\begin{equation*}
\begin{aligned}
&\frac{\Gamma_3(-w,-s,s+w+u)}{(2\pi)^{s+2w+3u}} {\Omega}_{k,s,w,u}(Z)\\[-0.1em]
=&\sum_{\{C,D\}}\Bigl(\sum_{T \in \mathcal{J}/GL_3(\mathbb{Z})}\frac{1}{\varepsilon_T}\sum_{U \in GL_3(\mathbb{Z})}E(T[U]|w,s,-s-w-u+2)e(T[U]Z) \Bigr)|_k [\bigl(\begin{smallmatrix}
     A_0&B_0\\
     C&D\\
   \end{smallmatrix}\bigr)]\\[-0.1em]
=&\sum_{T \in \mathcal{J}/GL_3(\mathbb{Z})} \frac{1}{\varepsilon_T} \sum_{U \in GL_3(\mathbb{Z})} E(T[U]| w,s,-s-w-u+2)\sum_{\{C,D\}} e(T[U]Z)|_k [\bigl(\begin{smallmatrix} 
    A_0&B_0\\
    C&D\\
   \end{smallmatrix}\bigr)]\\[-0.1em]
=&\sum_{T \in \mathcal{J}/GL_3(\mathbb{Z})}\frac{1}{\varepsilon_T} E(T| w,s,-s-w-u+2) \sum_{\substack{ \{C,D\} \\ U \in GL_3(\mathbb{Z})}} e(TZ)|_k [\bigl(\begin{smallmatrix}
     U&0\\
     0&^t\hspace{-0.02cm}{U}^{-1}\\
\end{smallmatrix}\bigr)\bigl(\begin{smallmatrix}
     A_0&B_0\\
     C&D\\
   \end{smallmatrix}\bigr)].\\[-0.1em]
\end{aligned}
\vspace{-0.2cm}
\end{equation*}
The inner sum is over a set of representatives for $\mathcal{N} \backslash Sp_3(\mathbb{Z})$, hence it is $2P_{k,T}$.
\end{proof}

Our next goal is to find a larger convergence region for the series on the right-hand side of Proposition~\ref{newrepresentation}.

\vspace{0.2cm}
\begin{lemma}\label{sum}
Let $k>6$. The series 
\vspace{-0.2cm}
\begin{equation*}
\sum_{T\in \mathcal{J}/GL_3(\mathbb{Z})} \frac{1}{\varepsilon_T}E(T| s,w,u)P_{k,T}(Z)
\vspace{-0.2cm}
\end{equation*}
is absolutely uniformly convergent on $K \times \tilde{K}$, where $K$ and $\tilde{K}$ are compact subsets of $\{ (s,w,u) \in \mathbb{C}^3/ \Re(s)>1, \Re(w)>1, \Re(u)>3-k/2\}$ and $\mathcal{H}_3$ respectively.
\end{lemma}

\begin{proof} [Proof] Since each compact subset of $\mathcal{H}_3$ is covered by finitely many images of $\mathcal{F}$ under the action of $Sp_3(\mathbb{Z})$ (see \cite[p.~31]{Kli}), we consider that $\tilde{K}$ is a compact subset of $\mathcal{F}$. Let $(s,w,u,Z)$ be an arbitrary element on $K \times \tilde{K}$. By Lemma~\ref{poincarebounded} and property (\ref{detE}) we have
\vspace{-0.05cm}
\begin{equation}\label{f eq 1.1}
\sum_{T\in \mathcal{J}/GL_3(\mathbb{Z})}\abs[\Big]{ \frac{1}{\varepsilon_T}E(T| s,w,u)P_{k,T}(Z)} \, \leq \, C_{\tilde{K},k} \sum_{T\in \mathcal{J}/GL_3(\mathbb{Z})} \abs[\big]{E(T| s,w,u-2+k/2)}
\vspace{-0.1cm}
\end{equation}
where $C_{\tilde{K},k}$ only depends on $\tilde{K}$ and $k$. In any term of the sum we can assume that $T$ belongs to $\mathcal{R}$, Minkowski's reduced domain. Then, from \cite[p.~51]{Mar3} and the bound of $E(I_3| \Re(s),\Re(w),\Re(u)-2+k/2)$ on $K$, there exist two constants $C,C_K>0$ such that
\vspace{-0.1cm}
\begin{equation}\label{f eq 1}
\begin{aligned}
\sum_{T\in \mathcal{J}/GL_3(\mathbb{Z})} & \abs[\Big]{E(T|s,w,u-2+k/2)} \leq C E(I_3| \Re(s),\Re(w),\Re(u)-2+k/2)\\[-0.7em]
 &\;\;\;\;\;\;\;\;\;\;\; \times \sum_{T\in \mathcal{J}/GL_3(\mathbb{Z})}(\det T_1)^{-\Re(s)} (\det T_2)^{-\Re(w)} (\det T)^{-\Re(u)+2-k/2}\\
 \hspace{0.1cm} & \leq \, C\, C_K\, \sum_{T\in \mathcal{J}/GL_3(\mathbb{Z})}(\det T_1)^{-\Re(s)} (\det T_2)^{-\Re(w)} (\det T)^{-\Re(u)+2-k/2}.\\[-0.2em]
\end{aligned}
\end{equation}

\vspace{-0cm}As $\mathcal{R}$ is a fundamental domain of $\mathcal{P}_3/GL_3(\mathbb{Z})$, $\mathcal{R}\cap \mathcal{J}$ is a discrete set of representatives of $\mathcal{J}/GL_3(\mathbb{Z})$. Hence, the entries of any $T \in \mathcal{R}\cap \mathcal{J}$ as in (\ref{notation}) satisfy $1 \leq t_1 \leq t_2 \leq t_3$, $0 \leq 2t_{12} \leq t_1, 2\abs{t_{13}} \leq t_1$ and $0\leq 2t_{23} \leq t_2.$
 Thus, the number of matrices $T=\left(\begin{smallmatrix} 
    t_1&*&*\\
    *&t_2&*\\
    *&*&t_3\\
   \end{smallmatrix}\right) \in \mathcal{R}\cap \mathcal{J}$ with fixed diagonal $t_1,t_2,t_3$ is at most $12t_1^2t_2$ (see \cite[p.~12,13]{Kli}). Furthermore, from the same reference, we know the existence of some $c_2,c_3>0$ such that
\vspace{-0.1cm}
\begin{equation*}
\begin{aligned}
\hfilneg (\det T_1)^{-\Re(s)}& (\det T_2)^{-\Re(w)} (\det T)^{-\Re(u)+2-k/2} \hspace{10000pt minus 1fil}\\
&\;\; \leq c_2^{-\Re(w)} \, c_3^{-\Re(u)+2-k/2} \, t_1^{-\Re(s+w+u)+2-k/2} \, t_2^{-\Re(w+u)+2-k/2} \, t_3^{-\Re(u)+2-k/2}
\end{aligned}
\vspace{-0.1cm}
\end{equation*}
for all $T=\left(\begin{smallmatrix} 
    t_1&*&*\\
    *&t_2&*\\
    *&*&t_3\\
   \end{smallmatrix}\right) \in \mathcal{R}\cap \mathcal{J}$. Next, we group matrices in $\mathcal{R}\cap \mathcal{J}$ that have the same diagonal elements, focusing on the last series in (\ref{f eq 1}), and considering a constant $D_{K,\tilde{K},k}$ such that $12c_2^{-\Re(w)} c_3^{-\Re(u)+2-k/2} C_{\tilde{K},k} C C_K\leq D_{K,\tilde{K},k} \text{ for all } (s,w,u) \in K,$ we have from (\ref{f eq 1.1}) and (\ref{f eq 1}) the inequalities
\vspace{-0.02cm}
\begin{equation*}
\begin{aligned}
&\sum_{T\in \mathcal{J}/GL_3(\mathbb{Z})} \abs[\Big]{ \frac{1}{\varepsilon_T}E(T| s,w,u)P_{k,T}(Z)} \\
& \leq D_{K,\tilde{K},k} \sum_{t_1 \leq t_2\leq t_3 \in \mathbb{N}}t_1^{-\Re(s+w+u)+4-k/2}t_2^{-\Re(w+u)+3-k/2}t_3^{-\Re(u)+2-k/2}\\
&\leq D_{K,\tilde{K},k} \sum_{n\in \mathbb{N}}n^{-\Re(s+w+u)+4-k/2}\sum_{n\in \mathbb{N}}n^{-\Re(w+u)+3-k/2}\sum_{n\in \mathbb{N}}n^{-\Re(u)+2-k/2}\\
&\leq D_{K,\tilde{K},k} \, \zeta(\Re(s+w+u)+k/2-4) \zeta(\Re(w+u)+k/2-3) \zeta(\Re(u)+k/2-2).
\vspace{0.1cm}
\end{aligned}
\end{equation*}
This last expression is bounded by a positive constant which only depends on the compact subsets $K, \tilde{K}$ and $k$, and the lemma is proved.
\end{proof}

Let us define two sets
\vspace{-0.1cm}
\begin{equation*}
\begin{array}{r@{}l}
C\,&=\{(s,w,u) \in \mathbb{C}^3 / \Re(s)>1, \Re(w)>1, \Re(s+w+u)<k/2-1\}\text{, and}\\
WC\,&=\{ (s,w,u) \in \mathbb{C}^3 / \Re(s)>1, \Re(w)>1, \Re(u)>k/2+1\}.\\
\end{array}
\vspace{-0.1cm}
\end{equation*}
\begin{proposition}\label{cuspACWC} 
Let $k>22$.
\begin{itemize}
\item[i)]{For any fixed $Z\in \mathcal{H}_3$, the function $(s,w,u) \mapsto {\Omega}_{k,s,w,u}(Z)$ is holomorphic on $A \cup C \cup WC$.}
\item[ii)]{For any fixed $(s,w,u) \in C$, ${\Omega}_{k,s,w,u}$ is a Siegel cusp form in $\mathfrak{S}_{3,k}$.}
\end{itemize}
\end{proposition}

\begin{proof} [Proof]
In order to start, we consider a map from Proposition~\ref{newrepresentation}
\vspace{-0.1cm}
\begin{equation}\label{repreC}
(s,w,u,Z)\hspace{-0.05cm} \mapsto \hspace{-0.05cm} \frac{2(2\pi)^{s+2w+3u}}{\Gamma_3(\hspace{-0.04cm}-w\hspace{-0.04cm},\hspace{-0.04cm}-\hspace{-0.04cm}s,s\hspace{-0.04cm}+\hspace{-0.04cm}w\hspace{-0.04cm}+\hspace{-0.04cm}u)} \hspace{-0.13cm} \sum_{T\in \mathcal{J}/GL_3(\mathbb{Z})} \hspace{-0.12cm}\frac{1}{\varepsilon_T}\hspace{-0.05cm}E(T| w,s,\hspace{-0.04cm}-\hspace{-0.04cm}s\hspace{-0.04cm}-\hspace{-0.04cm}w\hspace{-0.04cm}-\hspace{-0.04cm}u\hspace{-0.04cm}+\hspace{-0.04cm}2)\hspace{-0.04cm}P_{k,T}(Z).
\vspace{-0.1cm}
\end{equation}

{\it i)} Since Selberg’s Eisenstein series $E(T|w,s,-s-w-u+2)$ is holomorphic on $C$ for any $T \in \mathcal{J}$ (see remark after (\ref{def eis})), we obtain by Lemma~\ref{sum} that the function in (\ref{repreC}) is holomorphic on $C$. On the other hand, the set $B\cap C$ is a non-empty open subset of $A\cap C$ where the holomorphic functions (\ref{repreC}) and ${\Omega}_{k,s,w,u}(Z)$ are equal whenever $k>22$ (see Proposition~\ref{newrepresentation} and Lemma~\ref{AH3}). Thus, we conclude ${\Omega}_{k,s,w,u}(Z)$ has an analytic continuation to $C$ defined by (\ref{repreC}).

For $(s,w,u) \in WC$ we consider the map used in Lemma~\ref{wequation}
\vspace{-0.05cm}
\begin{equation}\label{def123}
(s,w,u) \mapsto e^{-(s+2w+3u)\pi i}{\Omega}_{k,w,s,-s-w-u+k}(Z),
\vspace{-0.05cm}
\end{equation}
where an analytic continuation of ${\Omega}_{k,s,w,u}(Z)$ from $A$ to $A \cup WA$ was obtained. The function in (\ref{def123}) is holomorphic on $WC$ by the conclusion in the paragraph above and the equivalence $(s,w,u) \in WC$ iff $(w,s,-s-w-u+k)\in C$, getting an analytic continuation of ${\Omega}_{k,s,w,u}(Z)$ from $A\cup WA \cup C$ to $A\cup WA \cup C \cup WC$ via (\ref{def123}). Both continuations agree on $WA\cap WC$ since $(s,w,u) \in WA \cap WC$ iff $(w,s,-s-w-u+k)\in A \cap C$ and the fact that ${\Omega}_{k,s,w,u}(Z)$ is well-defined on $A \cap C$; obtaining the first point {\it i)}. From now on in this work, and by technical reasons, we drop the set $WA$.

{\it ii)} Since the Poincaré series $P_{k,T}$ is holomorphic on $\mathcal{H}_3$ and invariant under the action of $Sp_3(\mathbb{Z})$ for any $T \in \mathcal{J}$ (see Subsection \ref{cusp}), we obtain due to Lemma~\ref{sum} that ${\Omega}_{k,s,w,u}$ defined by (\ref{repreC}) is holomorphic on $\mathcal{H}_3$ and invariant under the action of $Sp_3(\mathbb{Z})$ too.

Now we show ${\Omega}_{k,s,w,u}$ satisfies the cuspidal condition on $C$. We have to prove that it is annihilated by the Siegel $\Phi-$operator, i.e.
\vspace{-0.05cm}
\begin{equation}\label{siegelop}
\lim_{t \to \infty} {\Omega}_{k,s,w,u} \left(\bigl(\begin{smallmatrix} 
    Z_2&0\\
    0&it\\
   \end{smallmatrix}\bigr) \right)=0 \;\;\;\;\;\; \text{ for all }Z_2 \in \mathcal{H}_2, (s,w,u) \in C.
\vspace{-0.1cm}
\end{equation}
By Lemma~\ref{sum} we can apply the limit in the series (\ref{repreC}) term by term, and since $P_{k,T}$ is a cusp form for any $T \in \mathcal{J}$, we have $\lim_{t \to \infty} P_{k,T} \left(\bigl(\begin{smallmatrix}
    Z_2&0\\
    0&it\\
   \end{smallmatrix}\bigr) \right)=0$ for all $Z_2\in \mathcal{H}_2, \, T \in \mathcal{J}$. Thus (\ref{siegelop}) holds and we conclude the second point {\it ii)}.
\end{proof}

We want to find a convergence region for the series in Proposition~\ref{newrepresentation} with $-\varepsilon<\Re(w)<1+\varepsilon, \varepsilon>0$. To this end we define $\phi(s)=s(1-s)(s-\frac{1}{2})(\frac{1}{2}-s)$ and $\xi(2s)=\pi^{-s}\Gamma(s)\zeta(2s)$, and prove the following result.

\vspace{0.2cm}
\begin{lemma}\label{negativerealpart}
Let $k>6$. The series
\vspace{-0.2cm}
\begin{equation*}
\phi(s)\xi(2s) \sum_{T\in \mathcal{J}/GL_3(\mathbb{Z})} \frac{1}{\varepsilon_T}E(T| s,w,u)P_{k,T}(Z)
\vspace{-0.2cm}
\end{equation*}
is absolutely uniformly convergent on $K\times \tilde{K}$, where $K$ and $\tilde{K}$ are compact subsets of $\{(s,w,u) \in \mathbb{C}^3 / \Re(s)>-\frac{1}{2}, \Re(w)>\frac{5}{2}, \Re(u)>4-\frac k2\}$ and $\mathcal{H}_3$ respectively.
\end{lemma}

\begin{proof} [Proof]
For any $T\in \mathcal{J}/GL_3(\mathbb{Z})$ we notice that $E(T| s,w,u)$ can be written in terms of the Eisenstein series $E_s$ defined at Subsection \ref{eisenstein}, namely
\vspace{-0.1cm}
\begin{equation*}\label{E3Es}
(\det T)^u E(T| s,w,u) = E(T| s,w,0) = \sum_{A \in GL_3(\mathbb{Z})/P_{2,1}}E_s({T[A]_2}^0) (\det T[A]_2)^{-s/2-w}
\vspace{-0.15cm}
\end{equation*}
where $P_{2,1}=\left\{\bigl(\begin{smallmatrix} 
    U_2&*\\
    0&U_1\\
   \end{smallmatrix}\bigr) \in GL_3(\mathbb{Z}) / \  U_j \in GL_j(\mathbb{Z}) \right\}$ is another parabolic subgroup of $GL_3(\mathbb{Z})$ (see \cite[p.~243]{Ter2}), $Y^0 := (\det{Y})^{-1/2} \, Y \in \mathcal{SP}_2$ for any $Y \in {\mathcal P}_2$ and $E_s$ is evaluated using the identification between $\mathcal{SP}_2$ and $\mathcal{H}$ explained in (\ref{identification}). Next, we write $E_s\big({T[A]_2}^0\big)$ in terms of $\zeta_{\mathbb{Z}^2}^*$ using the relations in Subsection \ref{eisenstein}
\vspace{-0.1cm}
\begin{equation*}
\begin{aligned}
&2\zeta(2s)E_s\big({T[A]_2}^0\big)=2 \, \mathcal{Z}\big({T[A]_2}^0,s\big)=\Im({\tau_{{T[A]_2}^0}})^s \, \zeta_{\mathbb{Z}^2}(s,{\tau_{{T[A]_2}^0}})\\[-0.1em]
=&\, \Im({\tau_{{T[A]_2}^0}})^s \Bigl(2\zeta(2s)+\frac{2\pi^{1/2}\Gamma(s-1/2)\zeta(2s-1)}{\Gamma(s)}\Im(\tau_{{T[A]_2}^0})^{1-2s}+2\zeta^*_{\mathbb{Z}^2}(s,\tau_{{T[A]_2}^0})\Bigr).
\end{aligned}
\vspace{-0.05cm}
\end{equation*}
Therefore
\vspace{-0.05cm}
\begin{equation*}
\begin{aligned}
\phi(s)\xi(2s)&E(T| s,w,u)=(\det T)^{-u}\sum_{A \in GL_3(\mathbb{Z})/P_{2,1}}\Im(\tau_{{T[A]_2}^0})^s (\det T[A]_2)^{-s/2-w}\\[-0.6em]
&\times \left\{\phi(s)\xi(2s)+\phi(s)\xi(2s-1)\Im(\tau_{{T[A]_2}^0})^{1-2s}+\phi(s) \pi^{-s}\Gamma(s)\zeta^*_{\mathbb{Z}^2}(s,\tau_{{T[A]_2}^0}) \right\},
\end{aligned}
\vspace{-0.1cm}
\end{equation*}
and the series that we are going to study is
\vspace{-0.1cm}
\begin{equation}\label{asd}
\begin{aligned}
&\sum_{T\in \mathcal{J}/GL_3(\mathbb{Z})} \frac{1}{\varepsilon_T}(\det T)^{-u} P_{k,T}(Z) \sum_{A \in GL_3(\mathbb{Z})/P_{2,1}}\Im(\tau_{{T[A]_2}^0})^s (\det T[A]_2)^{-s/2-w}\\
&\;\;\;\;\;\times \left\{ \phi(s)\xi(2s)+\phi(s)\xi(2s-1)\Im(\tau_{{T[A]_2}^0})^{1-2s}+\phi(s) \pi^{-s}\Gamma(s)\zeta^*_{\mathbb{Z}^2}(s,\tau_{{T[A]_2}^0}) \right\}.
\end{aligned}
\end{equation}
\vspace{-0.2cm}

We prove the absolute convergence of (\ref{asd}) for a fixed $(s,w,u) \in K$ and an arbitrary $Z \in \tilde{K}$, and leave the case for the uniform convergence in compact subsets $K$ and $\tilde{K}$ to the reader, as it directly generalizes our proof.

The convergence of (\ref{asd}) holds when considering $Z$ within Siegel's fundamental domain $\mathcal{F}$, due to the invariance of $P_{k,T}$. Moreover, as explained in Lemma~\ref{sum}, $\tilde{K}$ can be assumed to be a compact subset of $\mathcal{F}$. Then, by Lemma~\ref{poincarebounded} there exists $C_{\tilde{K},k}>0$ such that $\abs[\big]{P_{k,T}(Z)} \leq C_{\tilde{K},k} (\det T)^{2-k/2}$ for all $T \in \mathcal{J}, Z \in \tilde{K}$. Thus, if we take the series of the absolute value in (\ref{asd}), use the facts above and recall that $\phi(s)\xi(2s)$ and $\phi(s)\xi(2s-1)$ are entire, we are able to reduce this proof to show the convergence of
\vspace{-0.05cm}
\begin{equation}\label{equation1}
\sum_{T\in \mathcal{J}/GL_3(\mathbb{Z})} \hspace{-0.5cm} (\det T)^{-\Re(u)+2-k/2} \hspace{-0.6cm} \sum_{A \in GL_3(\mathbb{Z})/P_{2,1}} \hspace{-0.5cm} \Im(\tau_{ {T[A]_2}^0})^{\Re(s)} (\det T[A]_2)^{-\Re(s/2+w)} \delta(T[A]_2,s)
\vspace{-0.02cm}
\end{equation}
where $\delta(L,s): \mathcal{P}_2 \to \mathbb{C}$ is one of the following three maps for any fixed $s$:
\vspace{-0.1cm}
\begin{equation*}
1, \quad \quad \Im(\tau_{L^0})^{1-2\Re(s)} \quad \quad \text{and} \quad \quad \abs[\big]{\zeta^*_{\mathbb{Z}^2}(s, \tau_{L^0})}.
\vspace{-0.15cm}
\end{equation*}

Before to work in (\ref{equation1}), we show here that any real power of $\Im(\tau_{{T[A]_2}^0})$ is bounded by a suitable power of $\det{\left(T[A]_2\right)}$. We can assume that the matrix $T[A]_2$ belongs to $\mathcal{R}_2 \cap \mathcal{J}_2$ for all $T \in \mathcal{J}/GL_3(\mathbb{Z}), A \in GL_3(\mathbb{Z})/P_{2,1}$ after a convenient choice of representative for $A$. Since for all $M=\bigl(\begin{smallmatrix}
     n&*\\
     *&m\\
 \end{smallmatrix}\bigr) \in \mathcal{R}_2 \cap \, \mathcal{J}_2$ one has $1\leq n\leq m$ and $C^{1/2}\leq (\det M)^{1/2}/n$ for some $C>0$ (see \cite[p.~13]{Kli}), one obtains
\vspace{-0.1cm}
\begin{equation*}
 C^{1/2}\leq \frac{ (\det T[A]_2)^{1/2}} {T[A]_{11}} = \Im(\tau_{{T[A]_2}^0}) \leq (\det T[A]_2)^{1/2}.
\vspace{-0.1cm}
\end{equation*}
Thus, we can conclude that for any $T \in \mathcal{J}/GL_3(\mathbb{Z}), A \in GL_3(\mathbb{Z})/P_{2,1}$
\vspace{-0.05cm}
\begin{equation}\label{imapower}
\Im(\tau_{{T[A]_2}^0})^l \leq \left\{ \;
\begin{array}{ll}
     \hspace{-0.2cm}(\det T[A]_2)^{l/2} & \text{if $l \geq 0$} \\
     \hspace{-0.2cm} C^{l/2} & \text{if $l<0$}.\\
\end{array} 
\right.
\vspace{-0.15cm}
\end{equation}

On the other hand, the estimate (\ref{zeta*}) yields the existence of a positive constant $C_s$ (which depends continuously on $s$) such that for all $\varepsilon >0$
\vspace{-0.05cm}
\begin{equation}\label{zetastarbound}
\abs[\big]{\zeta^*_{\mathbb{Z}^2}(s,\tau_{{T[A]_2}^0})}\leq C_s \, \Im(\tau_{{T[A]_2}^0})^{-\max \{1,2\Re(s)\}-\varepsilon}.
\end{equation}

Now we use (\ref{imapower}) and (\ref{zetastarbound}) to bound the three different inner series in (\ref{equation1})  (i.e. indexed by $A \in GL_3(\mathbb{Z})/P_{2,1}$). In the first case $\delta(L,s)=1$, we use (\ref{imapower}) and have
\vspace{-0.1cm}
\begin{equation}\label{bound1}
\begin{aligned}
\sum_{A}\Im(\tau_{ {T[A]_2}^0})^{\Re(s)} \, (\det T&[A]_2)^{-\Re(s/2+w)} \\[-0.6em]
&\leq \left\{
\begin{array}{ll}
   C^{\Re(s)/2} \sum_{A} (\det T[A]_2)^{-\Re(s/2+w)} & \text{ if } \Re(s)\leq0, \\
   \sum_{A} (\det T[A]_2)^{-\Re(w)}    & \text{ if } \Re(s)>0.\\
\end{array}
\right.
\end{aligned}
\end{equation}
For the second case $\delta(L,s)=\Im(\tau_{L^0})^{1-2\Re(s)}$, we use (\ref{imapower}) again and obtain
\vspace{-0.05cm}
\begin{equation}\label{bound2}
\begin{aligned}
\sum_{A}\Im(\tau_{ {T[A]_2}^0})^{1-\Re(s)} \, &(\det T[A]_2)^{-\Re(s/2+w)}\\[-0.6em]
& \leq \left\{
\begin{array}{ll}
  \sum_{A} (\det T[A]_2)^{1/2-\Re(s+w)} &\text{ if } \Re(s)\leq1, \\
  C^{1/2-\Re(s)/2} \sum_{A} (\det T[A]_2)^{-\Re(s/2+w)}  & \text{ if } \Re(s)>1.\\
\end{array} 
\right.
\end{aligned}
\vspace{-0.1cm}
\end{equation}
For the last case $\delta(L,s)=\abs[\big]{\zeta^*_{\mathbb{Z}^2}(s, \tau_{L^0})}$, we use (\ref{zetastarbound}) and then the identity $\Re(s)-\max \{1,\,2\Re(s)\}-\varepsilon<0$ for all $s\in \mathbb{C}$ in (\ref{imapower}) obtaining
\vspace{-0.05cm}
\begin{equation}\label{bound3}
\begin{aligned}
\sum_{A}\Im(\tau_{ {T[A]_2}^0})^{\Re(s)} \,(\det &\,T[A]_2)^{-\Re(w+s/2)} \, \abs[\big]{\zeta^*_{\mathbb{Z}^2}(s,\tau_{{T[A]_2}^0})} \\[-0.5em]
\leq & \, C_s \, C^{\Re(s)/2-\max \{1/2,\Re(s)\}-\varepsilon/2} \, \sum_{A} \, (\det T[A]_2)^{-\Re(s/2+w)}.
\end{aligned}
\vspace{-0.2cm}
\end{equation}

Applying the inequalities (\ref{bound1}), (\ref{bound2}) and (\ref{bound3}) to (\ref{equation1}) results in five series. From this fact we are able to conclude that the convergence of the series (\ref{equation1}) follows from the convergence of these three series:
\vspace{-0.1cm}
\begin{equation}\label{3series}
\begin{aligned}
 &\sum_{T} (\det T)^{-\Re(u)+2-k/2} \sum_{A} (\det T[A]_2)^{-\Re(w)}&&\; \text{ if } \Re(s) \geq 0,\\[-0.1em]
 &\sum_{T} (\det T)^{-\Re(u)+2-k/2} \sum_{A} (\det T[A]_2)^{1/2-\Re(s+w)} &&\; \text{if } \Re(s) \leq 1, \text{ and} \\[-0.1em]
 &\sum_{T} (\det T)^{-\Re(u)+2-k/2} \sum_{A} (\det T[A]_2)^{-\Re(s/2+w)} &&\; \text{for every } s,\\[-0.1em]
\end{aligned}
\vspace{-0.1cm}
\end{equation}
where the index $T$ (resp. $A$) runs throughout the quotient $\mathcal{J}/GL_3(\mathbb{Z})$ (resp. $GL_3(\mathbb{Z})/P_{2,1}$).

The three inner series in (\ref{3series}) is related to the following Hecke character
\vspace{-0.1cm}
\begin{equation*}
\mu_{r}(Y) := (\det Y)^{2r/3} \sum_{A \in GL_3(\mathbb{Z})/P_{2,1}} (\det Y[A]_2)^{-r}
\vspace{-0.1cm}
\end{equation*}
for $r>3/2, Y \in \mathcal{P}_3$. In \cite[p.~142]{Maa} Maass shows the existence of $C_r>0$ such that
\vspace{-0.1cm}
\begin{equation*}
\mu_r(Y)<C_r \left(\frac{\tr(Y)}{(\det Y)^{1/3}} \right)^r \mu_r(I_3)
\vspace{-0.1cm}
\end{equation*}
for all $Y \in \mathcal{P}_3$, where $C_r$ only depends on $r$ and $\mu_r(I_3)$ is convergent for $r>3/2$. Since $T \in \mathcal{J}/GL_3(\mathbb{Z}),$ we can assume $T \in \mathcal{R} \cap \mathcal{J}$. Then, for $T$ as in (\ref{notation}) we have $\tr (T)=t_1+t_2+t_3\leq t_1t_2t_3 \leq C \det T$ for some positive constant $C$ which does not depend on $T$ (see \cite[p.~13]{Kli}). The inequality of $\mu_r$ implies that our inner series satisfies
\begin{equation*}
\sum_{A \in GL_3(\mathbb{Z})/P_{2,1}} (\det T[A]_2)^{-r} < C^r \, C_r \, \mu_r(I_3) 
\vspace{-0.1cm}
\end{equation*}
for all $T \in \mathcal{J}/GL_3(\mathbb{Z})$, and that such bound only depends on $r$ whenever $r>3/2$.

Now we use this remark on the three inner series of (\ref{3series}). In the first case, if $\Re(s) \geq 0$ the inner series converges for $\Re(w)>3/2$. In the second case, if $\Re(s) \leq 1$ the inner series converges for $\Re(s+w)>2$. In the third case, the inner series converges for $\Re(s/2+w)>3/2$. Lemma's hypothesis fulfills these conditions.

Thus, the three series listed in (\ref{3series}) are bounded by some constants which depend continuously on $\Re(s)$ and $\Re(w)$ times the series
\vspace{-0.15cm}
\begin{equation*}
\sum_{T\in \mathcal{J}/GL_3(\mathbb{Z})} (\det T)^{-\Re(u)+2-k/2}.
\vspace{-0.2cm}
\end{equation*}
By \cite[p.~149]{Koh}, this series converges for $\Re(u)>4-k/2$.
\end{proof}

Since $\phi(w)\xi(2w)$ is entire, the function $\phi(w)\xi(2w){\Omega}_{k,s,w,u}(Z)$ satisfies the conclusions of Proposition~\ref{cuspACWC}. In the following lemma we show that such conclusions also hold on a set where $-\varepsilon<\Re(w)<1+\varepsilon, \varepsilon>0$. Let
\vspace{-0.1cm}
\begin{equation*}
D=\{(s,w,u) \in \mathbb{C}^3 / \Re(s)>5/2, \Re(w)>-1/2, \Re(s+w+u)<k/2-2 \}.
\vspace{-0.1cm}
\end{equation*}

\begin{lemma}\label{repreonD}
Let $k>22$. The function $\phi(w)\xi(2w){\Omega}_{k,s,w,u}(Z)$ admits an analytic continuation to $D \times \mathcal{H}_3$.
\end{lemma}

\begin{proof} [Proof]
By Proposition~\ref{cuspACWC}, the function $\phi(w)\xi(2w){\Omega}_{k,s,w,u}(Z)$ is holomorphic in the variables $(s,w,u,Z) \in C \times \mathcal{H}_3$ and is represented by
\begin{equation}\label{representationD}
\phi(w)\xi(2w)\frac{2(2\pi)^{s+2w+3u}}{\Gamma_3(-w,-s,s+w+u)}\hspace{-0.1cm}\sum_{T\in \mathcal{J}/GL_3(\mathbb{Z})}\hspace{-0.1cm}\frac{1}{\varepsilon_T}E(T| w,s,-s-w-u+2) P_{k,T}(Z).
\vspace{-0.05cm}
\end{equation}

Here we prove the series in (\ref{representationD}) is holomorphic separately in the variables $(s,w,u)\in D$ and $Z \in \mathcal{H}_3$. Firstly, we fix $Z$ in $\mathcal{H}_3$; for any $T \in \mathcal{J}$ the function $\phi(w)$$\xi(2w)$ $E(T| w,s,-s-w-u+2)$ is holomorphic on $D$ (see  \cite[p.~248]{Ter2}, Lemma~$1.5.3$), then by Lemma~\ref{negativerealpart} the function in (\ref{representationD}) is holomorphic on $D$. Now, we fix $(s,w,u) \in D$; the same argument works for the variable $Z$ because for any $T\in \mathcal{J}$ the function $P_{k,T}$ is holomorphic on $\mathcal{H}_3$, then by Lemma~\ref{negativerealpart} we obtain that (\ref{representationD}) is holomorphic on $\mathcal{H}_3$. That is, the function (\ref{representationD}) is holomorphic on $D \times \mathcal{H}_3$.
 
Since $C$ and $D$ have non-empty intersection, we conclude that $(s,w,u,Z) \mapsto$ $\phi(w)\xi(2w){\Omega}_{k,s,w,u}(Z)$ has an analytic continuation to $D \times \mathcal{H}_3$ defined by (\ref{representationD}).
\end{proof}

\begin{lemma}\label{Secondfunctionalequationa}
Let $k>22$. The series ${\Omega}_{k,s,w,u}(Z)$ satisfies
\begin{equation*}
\xi(2w){\Omega}_{k,s,w,u}(Z) = \xi(2-2w) {\Omega}_{k,s+w-1/2,1-w,w+u-1/2}(Z)
\end{equation*}
for all $(s,w,u,Z)\in D \times \mathcal{H}_3$ such that $-1/2<\Re(w)<0.$
\end{lemma}

\begin{proof} [Proof]
Lemma~\ref{repreonD} yields that the expression
$\xi(2w){\Omega}_{k,s,w,u}(Z)$ is represented by (\ref{representationD}) without the $\phi(w)$ factor. Then, using Selberg’s Eisenstein series' functional equation from Subsection \ref{defSelberg} one gets
\vspace{-0.1cm}
\begin{equation*}\label{secfuncteq1}
\begin{aligned}
&\; \frac{\Gamma_3(-w,-s,s+w+u)}{2(2\pi)^{s+2w+3u}} \xi(2w){\Omega}_{k,s,w,u}(Z)\\[-0.1em]
=&\; \xi(2w)\hspace{-0.15cm}\sum_{T\in \mathcal{J}/GL_3(\mathbb{Z})}\hspace{-0.1cm}\frac{1}{\varepsilon_T}E(T| w,s,-s-w-u+2)P_{k,T}(Z)\\[-0.1em]
=&\; \xi(2-2w)\hspace{-0.15cm}\sum_{T\in \mathcal{J}/GL_3(\mathbb{Z})}\hspace{-0.1cm}\frac{1}{\varepsilon_T}E(T|1-w,s+w-1/2,-s-w-u+2)P_{k,T}(Z)\\
=&\; \frac{\Gamma_3(w-1,1/2-s-w,s+w+u)}{2(2\pi)^{s+2w+3u}} \xi(2-2w){\Omega}_{k,s+w-1/2,1-w,w+u-1/2}(Z).
\end{aligned}
\vspace{-0.1cm}
\end{equation*}
The last equality results from defining $\xi(2-2w){\Omega}_{k,s+w-1/2,1-w,w+u-1/2}(Z)$ using (\ref{repreC}), as $(s+w-1/2,1-w,w+u-1/2) \in C.$ By the identity $\Gamma_3(-w,-s,s+w+u)=\Gamma_3(w-1,1/2-s-w,s+w+u)$ from (\ref{Gamma3property}), we obtain the lemma.
\end{proof}

We recall that the function $\phi(w)\xi(2w){\Omega}_{k,s,w,u}(Z)$ is holomorphic on $(A \cup C \cup WC \cup D) \times \mathcal{H}_3$ by Proposition~\ref{cuspACWC} and Lemma~\ref{repreonD}. For $k>14$ we define
\vspace{-0.1cm}
\begin{equation*}
\begin{array}{r@{}l}
aA\,&=\{(s,w,u)\in \mathbb{C}^3 / \Re(w)<0, \Re(s+w)>3/2, \Re(w+u)>9/2,\\
&\hspace{0.5cm}\Re(2s-w+u)<k-13/2\},\\
aC\,&=\{(s,w,u)\in \mathbb{C}^3 / \Re(w)<0, \Re(s+w)>3/2, \Re(s+w+u)<k/2-1\}, \text{ and}\\
aWC\,&=\{(s,w,u)\in \mathbb{C}^3 / \Re(w)<0, \Re(s+w)>3/2, \Re(w+u)>(k+3)/2\}.
\end{array}
\vspace{-0.1cm}
\end{equation*}

\begin{lemma}\label{asets}
Let $k>22$. The function $\phi(w)\xi(2w){\Omega}_{k,s,w,u}(Z)$ admits a holomorphic continuation in the variables $(s,w,u,Z)$ to $(aA \cup aC \cup aWC) \times \mathcal{H}_3$. 
\end{lemma}

\begin{proof} [Proof]
The set $aA$ (resp. $aC$ and $aWC$) is the image of the set $A$ (resp. $C$ and $WC$) under the involution $(s,w,u) \mapsto (s+w-1/2,1-w,w+u-1/2).$ Hence, we can define $\phi(w)\xi(2w){\Omega}_{k,s,w,u}(Z)$ on $(aA \cup aC \cup aWC) \times \mathcal{H}_3$ as a holomorphic function via
\vspace{-0.03cm}
\begin{equation}\label{ads}
\phi(w)\xi(2w){\Omega}_{k,s,w,u}(Z) = \phi(1-w)\xi(2-2w) {\Omega}_{k,s+w-1/2,1-w,w+u-1/2}(Z),
\vspace{-0.03cm}
\end{equation}
using that $\phi(1-w)\xi(2-2w)$ is entire and Proposition~\ref{cuspACWC}. On the other hand, the function on the left-hand side of (\ref{ads}) is holomorphic on $(A \cup C \cup WC \cup D) \times  \mathcal{H}_3$ as we pointed out after Lemma~\ref{Secondfunctionalequationa}. Additionally $aA \cup aC \cup aWC$ has a non-empty intersection with $A \cup C \cup WC \cup D$, in particular
\vspace{-0.05cm}
\begin{equation}\label{mnb}
 \{(s,w,u)\in D / -1/2<\Re(w)<0\} \subseteq aC \cap D.
\vspace{-0.05cm}
\end{equation}
Since the functional equation of Lemma~\ref{Secondfunctionalequationa} holds for the set on the left-hand side of (\ref{mnb}), one gets that $\phi(w)\xi(2w){\Omega}_{k,s,w,u}(Z)$ is well-defined and admits the claimed holomorphic continuation.
\end{proof}

In order to define ${\Omega}_{k,s,w,u}(Z)$ for $\Re(s)<0$, we introduce some sets for $k>14$
\vspace{-0.1cm}
\begin{equation*}
\begin{array}{r@{}l}
WaA\,&= \{(s,w,u)\in \mathbb{C}^3 / \Re(s)<0, \Re(s+w)>3/2, \Re(w+u)<k-9/2,\\
&\hspace{0.5cm}\Re(2s-w+u)>13/2\},\\
WaC\,&= \{(s,w,u)\in \mathbb{C}^3 / \Re(s)<0, \Re(s+w)>3/2, \Re(u)>k/2+1\},\\
WaWC\,&= \{(s,w,u)\in \mathbb{C}^3 / \Re(s)<0, \Re(s+w)>3/2, \Re(w+u)<(k-3)/2\}, \text{ and}\\
WD\,&= \{(s,w,u)\in \mathbb{C}^3 / \Re(s)>-1/2, \Re(w)>5/2, \Re(u)>k/2+2\}.\\
\end{array}
\vspace{-0.1cm}
\end{equation*}

\begin{lemma}\label{Wsets}
Let $k>22$. The function $\phi(s)\xi(2s)\phi(w)\xi(2w){\Omega}_{k,s,w,u}(Z)$ admits a holomorphic continuation in the variables $(s,w,u,Z)$ to $(WaA \cup WaC \cup WaWC \cup WD) \times \mathcal{H}_3.$
\end{lemma}

\begin{proof} [Proof]
The function $\phi(s)\xi(2s)\phi(w)\xi(2w){\Omega}_{k,s,w,u}(Z)$ is holomorphic on $(A \cup C \cup WC \cup D \cup aA \cup aC \cup aWC) \times \mathcal{H}_3$ by Lemma~\ref{asets}, the remark prior to Lemma~\ref{asets} and the holomorphic continuation of $\phi(s)\xi(2s)$ to $\mathbb{C}$. It is easy to see that $WaA$ (resp. $WaC, WaWC$ and $WD$) is the image of $aA$ (resp. $aC, aWC$ and $D$) under the involution $(s,w,u) \mapsto (w,s,-s-w-u+k).$ Hence, by Lemma~\ref{repreonD} and Lemma~\ref{asets}, we can define a holomorphic function on $(WaA \cup WaC \cup WaWC \cup WD) \times \mathcal{H}_3$ via
\vspace{-0.1cm}
\begin{equation}\label{estrella}
\begin{aligned}
\phi(s)\xi(2s)\phi(w)\xi(2w)&{\Omega}_{k,s,w,u}(Z)\\
&=\phi(s)\xi(2s)\phi(w)\xi(2w)e^{-(s+2w+3u)\pi i}{\Omega}_{k,w,s,-s-w-u+k}(Z).
\end{aligned}
\vspace{-0.05cm}
\end{equation}
The new set $WaA \cup WaC \cup WaWC \cup WD$ has a non-empty intersection with the old one $A \cup C \cup WC \cup D \cup aA \cup aC \cup aWC$, for example $WD \cap WC \neq \varnothing$. Since the function $\phi(s)\xi(2s)\phi(w)\xi(2w){\Omega}_{k,s,w,u}(Z)$ is well-defined in this last set (see (\ref{estrella}) and (\ref{def123})), the lemma is obtained.
\end{proof}

We define
\vspace{-0.13cm}
\begin{equation*}
aWD= \{(s,w,u)\in \mathbb{C}^3 / \Re(w)<-3/2, \Re(s+w)>0, \Re(w+u)>(k+5)/2\}.
\vspace{-0.1cm}
\end{equation*}

\begin{lemma}\label{aWD}
Let $k>22$. The function
\begin{equation}\label{defaWD}
\phi(s)\xi(2s)\phi(w)\xi(2w)\phi(s+w-1/2)\xi(2s+2w-1){\Omega}_{k,s,w,u}(Z)
\end{equation}
admits a holomorphic continuation in the variables $(s,w,u,Z)$ to $aWD \times \mathcal{H}_3$.
\end{lemma}

\begin{proof} [Proof]
We use Lemma~\ref{Wsets} to prove the facts in this proof. The function in (\ref{defaWD}) is holomorphic on the set $(A \cup C \cup WC \cup D \cup aA \cup aC \cup aWC \cup WaA \cup WaC \cup WaWC \cup WD) \times \mathcal{H}_3$. Since the set $aWD$ is the image of $WD$ under the involution $(s,w,u) \mapsto (s+w-1/2,1-w,w+u-1/2),$ we define the expression in (\ref{defaWD}) on $aWD \times \mathcal{H}_3$ as
\vspace{-0.05cm}
\begin{equation}\label{L12}
\phi(s)\xi(2s)\phi(1\hspace{-0.03cm}-\hspace{-0.03cm}w)\xi(2\hspace{-0.03cm}-\hspace{-0.03cm}2w)\phi(s+w-1/2)\xi(2s+2w-1){\Omega}_{k,s+w-1/2,1-w,w+u-1/2}(Z).
\vspace{-0.05cm}
\end{equation}
The function in (\ref{defaWD}) is holomorphic on $aWD \times \mathcal{H}_3$. Notice that $aWC \cap aWD\neq \varnothing$, then (\ref{defaWD}) is well-defined at this intersection since $(s,w,u) \in aWC \cap aWD$ iff $(s+w-1/2,1-w,w+u-1/2) \in WC \cap WD$ and the function $\phi(s')\xi(2s')\phi(w')\xi(2w')$ ${\Omega}_{k,s',w',u'}(Z)$ is well-defined for all $(s',w',u')$ on $WC \cup WD$.
\end{proof}
\begin{lemma}\label{thirdeq}
Let $k>22$. The series ${\Omega}_{k,s,w,u}(Z)$ satisfies the functional equation 
\vspace{-0.1cm}
\begin{equation*}
\begin{aligned}
\xi(2s)\xi(2w)\xi(2s+2w&-1){\Omega}_{k,s,w,u}(Z)\\
&=\xi(2-2s)\xi(2-2w)\xi(-2s-2w+3){\Omega}_{k,1-w,1-s,s+w+u-1}(Z)
\end{aligned}
\vspace{-0.1cm}
\end{equation*}
for all $(s,w,u,Z) \in aWD\times \mathcal{H}_3$ such that $3/2<\Re(s)<2.$
\end{lemma}

\begin{proof} [Proof]
Let us consider the following three expressions
\vspace{-0.1cm}
\begin{equation}\label{L} 
\phi(s)\xi(2s)\phi(w)\xi(2w)\phi(s+w-1/2)\xi(2s+2w-1){\Omega}_{k,s,w,u}(Z),
\end{equation}
\begin{equation}\label{I}
\begin{aligned}
\phi(s)\xi(2s)\phi(1-w)\xi(2-2w)\phi(-s-w+3/2)&\xi(-2s-2w+3)e^{-(s+2w+3u)\pi i}\\[-0.2em]
&\times {\Omega}_{k,s,-s-w+3/2,-u+k-1}(Z), \text{ and}
\end{aligned}
\end{equation}
\begin{equation}\label{R}
\begin{aligned}
\phi(1-s)\xi(2-2s)\phi(1-w)\xi(2-2w)\phi(-s-w+3/2&)\xi(-2s-2w+3)\\[-0.2em]
&\times {\Omega}_{k,1-w,1-s,s+w+u-1}(Z).
\end{aligned}
\end{equation}
To prove this lemma we show the equality of the holomorphic functions (\ref{L}) and (\ref{R}) with (\ref{I}); the $\phi$ factors can be canceled on both sides because the possible poles do not belong to the set in the hypothesis.

Lemma~\ref{aWD} establishes that the function in (\ref{L}) is defined by (\ref{L12}). Since $(s,w,u) \in aWD$ implies $(s+w-1/2,1-w,w+u-1/2) \in WD$, (\ref{L12}) is defined by
\vspace{-0.1cm}
\begin{equation*}
\begin{aligned}
\phi(s)\xi(2s)\phi(1-w)\xi(2-2w)\phi(s+w-1/2)\xi(2s+&2w-1)e^{-(s+2w+3u)\pi i}\\[-0.2em]
&\times{\Omega}_{k,1-w,s+w-1/2,-s-w-u+k}(Z)
\end{aligned}
\vspace{-0.1cm}
\end{equation*}
by Lemma~\ref{Wsets}. Note $(1-w,s+w-1/2,-s-w-u+k)\in D.$ Considering the hypothesis over $\Re(s)$ one concludes $0<\Re(s+w)<1/2$; then, we can apply Lemma~\ref{Secondfunctionalequationa} and obtain that (\ref{L}) is equal to (\ref{I}).

On the other hand, by the hypothesis it is easy to check $(1-w,1-s,s+w+u-1) \in aWC$. Hence, the function in (\ref{R}) is defined by
\vspace{-0.1cm}
\begin{equation*}
\phi(s)\xi(2s)\phi(1-w)\xi(2-2w)\phi(-s-w+3/2)\xi(-2s-2w+3) {\Omega}_{k,-s-w+3/2,s,w+u-1/2}(Z)
\vspace{-0.1cm}
\end{equation*}
according to Lemma~\ref{asets}. Note $(-s-w+3/2,s,w+u-1/2) \in WC$. By (\ref{def123}), this function is defined by (\ref{I}), proving that (\ref{R}) is equal to (\ref{I}) as desired.
\end{proof}

Before we prove Theorem~\ref{theoremanalyticcontinuation}, we introduce the alpha sets
\vspace{-0.1cm}
\begin{equation*}\label{defalpha}
\begin{array}{r@{}l}
\alpha A=&\, \{(s,w,u) \in \mathbb{C}^3 / \Re(s)<0, \Re(w)<0, 5<\Re(s+w+u)<k-3\}\\
\alpha C=&\, \{(s,w,u) \in \mathbb{C}^3 / \Re(s)<0, \Re(w)<0, \Re(u)<k/2-2\}\\
\alpha WC=&\, \{(s,w,u) \in \mathbb{C}^3 / \Re(s)<0, \Re(w)<0, \Re(s+w+u)>k/2+2\}\\
\alpha D=&\, \{(s,w,u) \in \mathbb{C}^3 / \Re(s)<3/2, \Re(w)<-3/2, \Re(u)<k/2-3\}\\
\alpha aA=&\, \{(s,w,u) \in \mathbb{C}^3 / \Re(s)>1, \Re(s+w)<1/2, 9/2<\Re(w+u)<k-7/2\}\\
\alpha aC=&\, \{(s,w,u) \in \mathbb{C}^3 / \Re(s)>1, \Re(s+w)<1/2, \Re(u)<k/2-2\}\\
\alpha aWC=&\, \{(s,w,u) \in \mathbb{C}^3 / \Re(s)>1, \Re(s+w)<1/2, \Re(w+u)>(k+3)/2\}\\
\alpha WaA=&\, \{(s,w,u) \in \mathbb{C}^3 / \Re(w)>1, \Re(s+w)<1/2, 7/2<\Re(w+u)<k-9/2\}\\
\alpha WaC=&\, \{(s,w,u) \in \mathbb{C}^3 / \Re(w)>1, \Re(s+w)<1/2, \Re(s+w+u)>k/2+2\}\\
\alpha WaWC=&\, \{(s,w,u) \in \mathbb{C}^3 / \Re(w)>1, \Re(s+w)<1/2, \Re(w+u)<(k-3)/2\}\\
\alpha WD=&\, \{(s,w,u) \in \mathbb{C}^3 / \Re(s)<-3/2, \ \Re(w)<3/2, \ \Re(s+w+u)>k/2+3\}.
\end{array}
\end{equation*}

\begin{proof} [Proof\nopunct] {\it of Theorem~\ref{theoremanalyticcontinuation}.}
Firstly, we show that the function in (\ref{finalfunction}) has an analytic continuation to the union of the alpha sets times $\mathcal{H}_3$. Let us call $A\cup C \cup WC \cup D \cup aA \cup aC \cup aWC \cup WaA \cup WaC \cup WaWC \cup WD \cup aWD \subset \mathbb{C}^3$ the old set. The function in (\ref{finalfunction}) is holomorphic on the old set times $\mathcal{H}_3$ by Lemma~\ref{aWD}. Notice that each alpha set is the image of the corresponding set (that without the alpha letter in front of its name) under the involution $(s,w,u) \mapsto (1-w,1-s,s+w+u-1)$; then Lemma~\ref{Wsets} allows us to define a holomorphic function for (\ref{finalfunction}) on the union of the alpha sets times $\mathcal{H}_3$ via
\vspace{-0.05cm}
\begin{equation}\label{defanalcont}
\phi(1\hspace{-0.03cm}-\hspace{-0.03cm}s)\xi(2\hspace{-0.03cm}-\hspace{-0.03cm}2s)\phi(1\hspace{-0.03cm}-\hspace{-0.03cm}w)\xi(2\hspace{-0.03cm}-\hspace{-0.03cm}2w)\phi(-\hspace{-0.03cm}s-\hspace{-0.03cm}w\hspace{-0.03cm}+\hspace{-0.03cm}3/2)\xi(\hspace{-0.03cm}-2s\hspace{-0.03cm}-\hspace{-0.03cm}2w\hspace{-0.03cm}+\hspace{-0.03cm}3){\Omega}_{k,1\hspace{-0.03cm}-\hspace{-0.03cm}w,1\hspace{-0.03cm}-\hspace{-0.03cm}s,s\hspace{-0.03cm}+\hspace{-0.03cm}w\hspace{-0.03cm}+\hspace{-0.03cm}u\hspace{-0.03cm}-\hspace{-0.03cm}1}(Z)
\vspace{-0.05cm}
\end{equation}
(recall $\phi(s)\xi(2s)$ is entire). Furthermore, the old set and the union of the alpha sets have a non-empty intersection, in particular
\vspace{-0.1cm}
\begin{equation}\label{intersection56}
aWD \cap \{(s,w,u) \in \mathbb{C}^3/ \ 3/2<\Re(s)<2\} \subseteq aWD \cap \alpha aWC.
\vspace{-0.1cm}
\end{equation}
By Lemma~\ref{thirdeq}, the functions in (\ref{finalfunction}) and (\ref{defanalcont}) are equal on the region on the left-hand side of (\ref{intersection56}) (recall the identity $\phi(s)=\phi(1-s)$). Hence, the analytic continuation of (\ref{finalfunction}) to the union of the alpha sets times $\mathcal{H}_3$ via (\ref{defanalcont}) is well-defined.

An open set $X \subseteq \mathbb{C}^3$ is called a \emph{tube} if there exists an open set $R \subseteq \mathbb{R}^3$ such that $X$ can be written as $\{z \in \mathbb{C}^3/ \ \Re(z) \in R\}.$ There is a classical result in the theory of several complex variables which establishes that a holomorphic function on a connected tube has a holomorphic continuation to its convex hull (see for instance \cite[p.~41]{Hor}). It is not hard to check that the union of the old set with the alpha sets is a connected tube in $\mathbb{C}^3$. Thus, the result mentioned yields that the function in (\ref{finalfunction}) admits a holomorphic continuation to the convex hull of such a tube, which is $\mathbb{C}^3$. This ends the proof of the first claim of the theorem.

Now we prove that the analytic continuation of (\ref{finalfunction}) is invariant under $Sp_3(\mathbb{Z})$. For an arbitrary $M \in Sp_3(\mathbb{Z})$, we consider the function initially defined on $A \times \mathcal{H}_3$
\vspace{-0.05cm}
\begin{equation*}
\phi(s)\xi(2s)\phi(w)\xi(2w)\phi(s+w-1/2)\xi(2s+2w-1) \left( {\Omega}_{k,s,w,u}(Z)|_k[M]-{\Omega}_{k,s,w,u}(Z)\right).
\vspace{-0.05cm}
\end{equation*}
It has a holomorphic continuation to $\mathbb{C}^3 \times \mathcal{H}_3$ by the first part of this proof; and it vanishes on $C \times \mathcal{H}_3$ by Proposition~\ref{cuspACWC}. Then, such a holomorphic continuation must be zero on $\mathbb{C}^3 \times \mathcal{H}_3$ for any $M \in Sp_3(\mathbb{Z})$; hence, the holomorphic continuation of (\ref{finalfunction}) is invariant under the action of $Sp_3(\mathbb{Z})$.

Since the analytic continuation of (\ref{finalfunction}) is a Siegel modular form of weight $k$ over $Sp_3(\mathbb{Z})$, it has a Fourier series representation
\vspace{-0.1cm}
\begin{equation*}
\sum_{T \geq 0} A_Te^{2\pi i \tr(TZ)}
\vspace{-0.15cm}
\end{equation*}
where $T$ runs over the set of half-integral, semi-positive-definite, 3 by 3 matrices (see for instance \cite[p.~46]{Kli}). For each $T\geq 0$, we consider the holomorphic function $\Phi_T:\mathbb{C}^3 \to \mathbb{C}$ defined by $\Phi_T(s,w,u)=A_T(s,w,u)$. Recalling (\ref{fourier}) and using Proposition~\ref{cuspACWC}, we know that $\Phi_T$ vanishes on $C$ for any non-positive-definite $T$, so $\Phi_T$ is the zero function in all these cases. In other words, for any fixed $(s,w,u) \in \mathbb{C}^3$ the Fourier coefficients of the analytic continuation of (\ref{finalfunction}) indexed by the non-positive-definite matrices $T$ are zero. These facts yield the second claim of this theorem.

Finally, the functional equations follow from the analytic continuation just proved and Lemmas \ref{wequation}, \ref{Secondfunctionalequationa} and \ref{thirdeq}.
\end{proof}

We prove a new functional equation for our kernel.

\vspace{0.2cm}
\begin{lemma}\label{fourtheq}
Let $k>22$. The series ${\Omega}_{k,s,w,u}(Z)$ satisfies
\vspace{-0.05cm}
\begin{equation}\label{bequation}
\xi(2s){\Omega}_{k,s,w,u}(Z) = \xi(2-2s) {\Omega}_{k,1-s,s+w-1/2,u}(Z) 
\vspace{-0.05cm}
\end{equation}
for all $(s,w,u,Z) \in WD\times \mathcal{H}_3$ such that $-1/2<\Re(s)<0.$
\end{lemma}

\begin{proof} [Proof]
This proof is similar to the one of Lemma~\ref{Secondfunctionalequationa}, so we omit some details. We observe that $\Gamma_3(-s,-w,-u+k)=\Gamma_3(s-1,-s-w+1/2,-u+k)$ and $(s,w,u)\in WD$ iff $(w,s,-s-w-u+k) \in D$. The functional equation in Subsection \ref{defSelberg} implies
\vspace{-0.05cm}
\begin{equation*}
\xi(2s) E(T| s,w,u+2-k)=\xi(2-2s)E(T| 1-s,s+w-1/2,u+2-k).
\vspace{-0.05cm}
\end{equation*}
Using the previous remarks, (\ref{estrella}), (\ref{representationD}) and the hypothesis, we conclude that the left-hand side of (\ref{bequation}) times $\phi(s) \phi(w) \xi(2w)$ can be represented by
\vspace{-0.02cm}
\begin{equation}\label{bintermediateequation}
\begin{aligned}
\phi(s) \xi(2-2s) \phi(w)&\xi(2w) e^{-(s+2w+3u)\pi i} \frac{2(2\pi)^{-(s+2w+3u-3k)}}{\Gamma_3(s-1,-s-w+1/2,-u+k)}\\[-0.3em]
\times& \sum_{T\in \mathcal{J}/GL_3(\mathbb{Z})} \frac{1}{\varepsilon_T}E(T| 1-s,s+w-1/2,u+2-k)P_{k,T}(Z).
\end{aligned}
\end{equation}

\vspace{-0.15cm}On the other hand, we observe that the hypothesis on $(s,w,u)$ implies $(1-s,s+w-1/2,u)\in WC$. Then, using (\ref{def123}) and (\ref{repreC}) one can see that the right-hand side of (\ref{bequation}) times $\phi(1-s) \phi(w) \xi(2w)$ can be represented by (\ref{bintermediateequation}) as well. This proves the identity (\ref{bequation}) multiplied by $\phi(s) \phi(w) \xi(2w)$. Now we divide by these factors because they have no zeros in the region given in the hypothesis.
\end{proof}

\begin{proof} [Proof\nopunct] {\it of Proposition~\ref{group}.}
Let us consider the following maps from $\mathbb{C}^3$ to $\mathbb{C}^3$
\vspace{-0.14cm}
\begin{equation*}
\begin{array}{r@{}l}
w&:(s,w,u)\mapsto(w,s,-s-w-u+k) \\
a&:(s,w,u)\mapsto(s+w-1/2,1-w,w+u-1/2) \\
aba&:(s,w,u)\mapsto(1-w,1-s,s+w+u-1) \\
b&:(s,w,u)\mapsto(1-s,s+w-1/2,u).
\end{array}
\vspace{-0.12cm}
\end{equation*}
They appear in the functional equations of Theorem~\ref{theoremanalyticcontinuation} and Lemma~\ref{fourtheq} respectively, and belong to the group of bijections from $\mathbb{C}^3$ to $\mathbb{C}^3$ whose operation is the composition of functions. Every composition of these functions yield a functional equation for ${\Omega}_{k,s,w,u}(Z)$. For example, the composition of $a$ and $w$ is the map $aw:(s,w,u)\mapsto(s+w-1/2,1-s,-w-u+k-1/2)$, which corresponds to the functional equation
\vspace{-0.05cm}
\begin{equation*}
e^{(s+2w+3u)\pi i} \, \xi(2s) {\Omega}_{k,s,w,u}(Z)=\xi(2-2s) {\Omega}_{k,s+w-1/2,1-s,-w-u+k-1/2}(Z);\\
\vspace{-0.05cm}
\end{equation*}
this is obtained using the first functional equation of Theorem~\ref{theoremanalyticcontinuation} named here $w$, and then the second one named $a$.

Let us define $G=\langle a,aba, w \rangle$ as the group generated by the functional equations of Theorem~\ref{theoremanalyticcontinuation} under composition. If $aw \in G$ is raised to the sixth power one gets the identity function; in fact, the order of $aw$ is six. In the same way one checks that the order of $a,b,aba \text{ and }w$ is two, then $b=a\,aba\,a \in G.$ It is not hard to see that $waw^{-1}=b$ and $wbw^{-1}=a$, so $b\,aw\,b^{-1}=bw=(aw)^{-1}.$ Hence $\widetilde{G}:=\langle aw,b/\,(aw)^6,b^2,bawb^{-1}aw \rangle$ is a subgroup of $G$ isomorphic to $D_{12}.$ On the other hand, $a=(aw)^2b$ and $w=a\hspace{.08cm}aw$; therefore $a,aba,w \in \widetilde{G}$ and so $G=\widetilde{G}\cong D_{12}.$ 
\end{proof}

\begin{remark}
Since the group $G$ in the proof of Proposition~\ref{group} can be presented in different forms, we can establish the corresponding functional equations of ${\Omega}_{k,s,w,u}(Z)$ in different ways. For example, we exhibit two
\vspace{-0.05cm}
\begin{equation*}
\begin{aligned}
e^{(s+2w+3u)\pi i}{\Omega}_{k,s,w,u}(Z)=&\;{\Omega}_{k,w,s,-s-w-u+k}(Z)\\
\xi(2s) {\Omega}_{k,s,w,u}(Z)=&\; \xi(2-2s){\Omega}_{k,1-s,s+w-1/2,u}(Z)\\
\xi(2w){\Omega}_{k,s,w,u}(Z) =&\; \xi(2-2w) {\Omega}_{k,s+w-1/2,1-w,w+u-1/2}(Z)
\end{aligned}
\vspace{-0.1cm}
\end{equation*}
\begin{equation*}
\begin{aligned}
e^{(s+2w+3u)\pi i} \xi(2s){\Omega}_{k,s,w,u}(Z)=&\; \xi(2-2s) {\Omega}_{k,s+w-1/2,1-s,-w-u+k-1/2}(Z)\\
\xi(2s) {\Omega}_{k,s,w,u}(Z) =&\; \xi(2-2s) {\Omega}_{k,1-s,s+w-1/2,u}(Z).
\end{aligned}
\vspace{0.2cm}
\end{equation*}
The first presentation of $G$ is $\langle w,b,a \rangle$ and the second one $\widetilde{G}=\langle aw, b \rangle$.
\end{remark}

\section{Several variables Dirichlet series associated with Siegel cusp forms}\label{SectionKM}

Throughout this section $f$ denotes a Siegel cusp form in $\mathfrak{S}_{3,k}$ with Fourier series representation (\ref{fourier}). The Koecher-Maass series associated with $f$ is (\ref{KSdirichlet}) and its twist by Selberg's Eisenstein series $E$ is
\begin{equation*}\label{detT-l}
\sum_{T \in \mathcal{J}/SL_3(\mathbb{Z})} \frac{1}{\varepsilon_T} E(T|s,w,u) \frac{A_T}{(\det T)^l},
\vspace{-0.05cm}
\end{equation*}
where $s,w,u \text{ and } l$ are complex numbers lying in some set where the series converges. By (\ref{detE}), this series is the following Dirichlet series in three complex variables.

\vspace{0.2cm}
\begin{definition}\label{defidirichlet}
The Koecher-Maass series twisted by the Selberg’s Eisenstein series is
\vspace{-0.05cm}
\begin{equation*}
D_f(s,w,u)=\sum_{T \in \mathcal{J}/SL_3(\mathbb{Z})} \frac{A_T}{\varepsilon_T}E(T|s,w,u).
\vspace{-0.1cm}
\end{equation*}
\end{definition}
This is not just a formal expression, it has a convergence region that we determine after a technical lemma.

\vspace{0.2cm}
\begin{lemma}\label{serieE} The series
\vspace{0.2cm}
\begin{equation*} 
\sum_{T\in \mathcal{J}/GL_3(\mathbb{Z})} E(T|s,w,u)
\vspace{0.15cm}
\end{equation*}
is absolutely uniformly convergent in compact subsets of $\{(s,w,u) \in \mathbb{C}^3 /  \Re(s)>1, \Re(w)>1, \Re(u)>1\}$
and equal to $\sum_{T\in \mathcal{J}/SL_3(\mathbb{Z})} E(T|s,w,u).$
\end{lemma}

\begin{proof} [Proof]
Let $K$ be a compact subset of the set of the hypothesis. From the proof of Lemma~\ref{sum} one can obtain the inequality
\vspace{-0.05cm}
\begin{equation*}
\begin{array}{r@{}l}
\displaystyle\sum_{T\in \mathcal{J}/GL_3(\mathbb{Z})} \abs[\big]{E(T|s,w,u-2+k/2)}\leq & \, C_{K,k} \,\zeta(\Re(s+w+u)+k/2-4) \\[-0.4em]
&\;\;\;\; \times \zeta(\Re(w+u)+k/2-3) \, \zeta(\Re(u)+k/2-2)
\end{array} 
\end{equation*}
whenever $\Re(s)>1, \Re(w)>1$ and $\Re(u)>3-k/2$, where $C_{K,k}$ is a positive constant which only depends on the compact set $K$ and the weight $k$. Hence, changing $u$ by $u+2-k/2$ one obtains the first statement of this lemma.

It is easy to see $\mathcal{J}/GL_3(\mathbb{Z})=\mathcal{J}/SL_3(\mathbb{Z})$, so the series in the lemma are equal.
\end{proof}

We prove that a convergence set for our twisted Dirichlet series is
\vspace{-0.1cm}
\begin{equation*}
    I=\{(s,w,u)\in \mathbb{C}^3 / \Re(s)>1, \Re(w)>1, \Re(u)>k/2+1 \}.
\vspace{-0.1cm}
\end{equation*}
\begin{lemma}\label{Dirichletconverges}
The twisted Koecher-Maass series of Definition~\ref{defidirichlet} is absolutely convergent and defines a holomorphic function on $I$, for any $f \in \mathfrak{S}_{3,k}.$
\end{lemma}

\begin{proof} [Proof]
Hecke's estimate establishes that the Fourier coefficients of any $f$ in $\mathfrak{S}_{3,k}$ satisfy $A_T=\mathcal{O} ((\det T)^{k/2})$, where the constant depends on $f$ (see \cite[p.~149]{Koh}). Furthermore, (\ref{detE}) imply that the series in Definition~\ref{defidirichlet} satisfies
\vspace{-0.1cm}
\begin{equation*}
\sum_{T \in \mathcal{J}/SL_3(\mathbb{Z})} \abs[\Big]{\frac{A_T}{\varepsilon_T}E(T|s,w,u)} \; \leq \;C_f \sum_{T \in \mathcal{J}/SL_3(\mathbb{Z})} \abs[\big]{E(T|s,w,u-k/2)}
\vspace{-0.05cm}
\end{equation*}
for some $C_f>0$. Since for any $T \in \mathcal{J}$ the Selberg's Eisenstein series $E(T|s,w,u)$ is a holomorphic function whenever $\Re(s)>1, \Re(w)>1$ (see Subsection \ref{defSelberg}), and applying Lemma~\ref{serieE} to the previous series, we obtain the lemma.
\end{proof}

Before going to the proof of Theorem~\ref{teorema2}, we prove a technical lemma and then bound the Poincaré series $P_{k,T}$ over Siegel's fundamental domain $\mathcal{F}$.

\vspace{0.2cm}
\begin{lemma}\label{A_T}
There exists a constant $C>0$ such that
\begin{equation*}
\abs[\big]{A_T} \, \leq \, C \, \sup_{Z \in \mathcal{F}}\abs[\big]{f(Z)} \, (\det T)^k
\vspace{-0.2cm}
\end{equation*}
for all cusp form $f$ in $\mathfrak{S}_{3,k}$.
\end{lemma}

\begin{proof} [Proof]
Lemma~1 in \cite[p.~143]{Kli} states that the Fourier coefficients of any $f$ in $\mathfrak{S}_{3,k}$ satisfy $\abs{A_T}\leq C\, (\det T)^k$, with a constant $C$ which is independent of $T$ but depends on $f$. In the middle of that proof the author proves the inequality
\vspace{-0.05cm}
\begin{equation*}
\abs[\big]{f(Z)} \, \leq \, D \, (1+\tr(Y))^{3k} \, (\det Y)^{-k} \, \abs[\big]{f(Z_1)},
\vspace{-0.05cm}
\end{equation*}
where $Z=X+iY$ is an arbitrary element in $\mathcal{H}_3$, $Z_1$ is a representative of $Z$ in $\mathcal{F}$ and $D$ is a positive constant independent of $Z$ and $f$. Furthermore, recalling the existence of a $\delta>0$ such that $\mathcal{F} \subseteq V(\delta) \subseteq \{Z \in \mathcal{H}_3 / \ \Im(Z)\geq \delta I_3\}$, one knows by Koecher's theorem that every $f \in \mathfrak{S}_{3,k}$ is bounded on $\mathcal{F}$ (see \cite[p.~30,45]{Kli}); hence $\{\abs{f(Z_{1})} \, / \, Z_{1} \in \mathcal{F} \} $ is a not-empty, bounded subset of the real numbers. Thus $\sup_{Z_{1} \in \mathcal{F}}\abs{f(Z_{1})}$ exists and then one has
\vspace{0.1cm}
\begin{equation}\label{estimate}
\abs[\big]{f(Z)} \, \leq \, D \, (1+\tr(Y))^{3k} \, (\det Y)^{-k} \, \sup_{Z_1 \in \mathcal{F}} \abs[\big]{f(Z_1)}
\vspace{0.1cm}
\end{equation}
for all $Z \in \mathcal{H}_3$. On the other hand, we consider the integral formula for the Fourier coefficients of $f$
\vspace{0.1cm}
\begin{equation*}
    A_T=\int_{X\text{ mod }1}f(Z)\,e^{-2\pi i \tr(TZ)}\,dX
\vspace{0.25cm}
\end{equation*}
with $dX=dx_1\cdot...\cdot dx_6$ for $X =\Re{(Z)}$ as in (\ref{notation}) (see for instance \cite[p.~44]{Kli}). We apply the estimate (\ref{estimate}) to the integral formula with $Z=X+iT^{-1}$, and obtain
\vspace{-0.05cm}
\begin{equation}\label{coefficientAT}
\begin{aligned}
\abs[\big]{A_T} &\leq \int_{X\text{ mod }1}\abs[\big]{f(X+iT^{-1})\,e^{-2\pi i \tr(T(X+iT^{-1}))}}\,dX \\[0.2em]
&\leq D\, e^{6\pi}\, \sup_{Z_1 \in \mathcal{F}}\abs[\big]{f(Z_1)} \, (\det T)^k \,(1+\tr(T^{-1}))^{3k}.
\end{aligned}
\vspace{-0.05cm}
\end{equation}
We know $A_{T[U]}=(\det U)^k A_T$ for all $U \in GL_3(\mathbb{Z})$ (\cite[p.~45]{Kli}), so we may assume $T \in \mathcal{J}\cap \mathcal{R}.$ By Lemma~2 in \cite[p.~20]{Kli} there exists a positive number $\gamma$ such that
\vspace{-0.1cm}
\begin{equation*}
    0<T^D<\gamma \, T \,\,\;\;\;\;\\\\\\ \text{ for all } T \in \mathcal{J}\cap \mathcal{R},
\vspace{-0.15cm}
\end{equation*}
where $T^D$ is the diagonal matrix made up of the diagonal elements $t_1, t_2, t_3$ of the matrix $T$ as in (\ref{notation}). By Corollary $7.7.4.(a)$ and $(d)$ in \cite[p.~495]{Horn} one has
\begin{equation*}
\gamma^{-1}\,T^{-1}<(T^{D})^{-1} \text{ and then } \tr(T^{-1})<\gamma \tr((T^{D})^{-1})\leq \gamma (t_1^{-1}+t_2^{-1}+t_3^{-1}) \leq 3 \gamma.    
\vspace{-0.05cm}
\end{equation*}
Considering $\tr(T^{-1})\leq3 \, \gamma$ in (\ref{coefficientAT}), we obtain the claimed.
\end{proof}

For $k>6$ we consider next the function $P^k_3:\mathcal{H}_3 \times \mathcal{H}_3 \to \mathbb{C}$ defined by
\vspace{-0.02cm}
\begin{equation*}
P^k_3(Z,Z_2) \ = \sum_{M \in \{\pm I_6\} \backslash Sp_3(\mathbb{Z})} j(M,Z)^{-k} \det(MZ+Z_2)^{-k}.
\vspace{-0.05cm}
\end{equation*}
One knows that for any fixed $Z \in \mathcal{H}_3,$ the function $Z_2 \mapsto P^k_3(Z,Z_2)$ is a Siegel cusp form of weight $k$ over $Sp_3(\mathbb{Z})$ (see \cite[p.~78,79,90]{Kli}). Its Fourier series representation is
\begin{equation*}
P^k_3(Z,Z_2)=b_k^{-1} \sum_{T \in \mathcal{J}} (\det T)^{k-2} P_{k,T}(Z) \,e^{2\pi i \tr(T Z_2)},
\vspace{-0.1cm}
\end{equation*}
where $b_k=(4\pi)^{3/2}\,(-2\pi i)^{-3k}\,\Gamma(k)\,\Gamma(k-1/2) \, \Gamma(k-1)$ and $P_{k,T}$ the Poincaré series defined in Subsection \ref{cusp}. We consider an arbitrary $Z \in \mathcal{F}$ and apply Lemma~\ref{A_T} to the $T^{th}$-Fourier coefficient of this function, obtaining the existence of a positive constant $C$ (which does not depend on $Z$) such that
\begin{equation*}
\abs[\big]{P_{k,T}(Z)}\leq C \;\abs{b_k} (\det T)^2 \sup_{Z_2 \in \mathcal{F}}\abs[\big]{P^k_3(Z,Z_2)} 
\vspace{-0.15cm}
\end{equation*}
for all $Z \in \mathcal{F}$. On the other hand, the function $P^k_3(Z,Z_2) (\det \Im(Z) \, \det \Im(Z_2))^{k/2}$ is bounded on $\mathcal{F} \times \mathcal{F}$ (see \cite[p.~83]{Kli}); hence, there exists a positive constant $C_2$ (which is independent of $Z$) such that
\begin{equation*}
\abs[\big]{P_{k,T}(Z)}\leq C \;C_2\;\abs{b_k}\,(\det T)^2 \sup_{Z,Z_2 \in \mathcal{F}} (\det \Im(Z) \, \det \Im(Z_2))^{-k/2} .
\vspace{-0.08cm}
\end{equation*}
Let us consider $\delta>0$ such that $\mathcal{F}\subseteq V(\delta)$. For $Z \in \mathcal{F}$ we have $Y=\Im(Z)$ belongs to Minkowski's reduced domain and $Y-\delta I_3 \geq 0$, by definition of $V(\delta)$. Consequently, there exists a positive constant $C_3$ (which does not depend on $Y$) such that
\vspace{-0.03cm}
\begin{equation}\label{dety>delta}
    \det Y\geq C_3 \, y_1 \, y_2 \, y_3\geq C_3 \, \delta^3
\vspace{-0.05cm}
\end{equation}
with $Y$ as in (\ref{notation}) (see \cite[p.~13]{Kli}). Then $\sup_{Z,Z_2 \in \mathcal{F}}\ (\det \Im(Z) \, \det \Im(Z_2))^{-k/2}\leq C_3^{-k} \, \delta^{-3k}$ and we finally get
\vspace{-0.08cm}
\begin{equation}\label{poincareT}
\abs[\big]{P_{k,T}(Z)}\leq C_k (\det T)^2 \;\; \qquad \text{ for all } Z \in \mathcal{F}, T \in \mathcal{J}
\vspace{-0.05cm}
\end{equation}
where $C_k>0$ is a constant depending on $k$.

\begin{proof} [Proof\nopunct] {\it of Theorem~\ref{teorema2}.}
First, we prove that the series
\vspace{-0.13cm}
\begin{equation}\label{v}
\sum_{T \in \mathcal{J}/SL_3(\mathbb{Z})} \int_{Z \in \mathcal{F}} \abs[\Big]{\frac{1}{\varepsilon_T}E(T|w,s,-s-w-u+2) \, P_{k,T}(Z) \, \overline{f^*(Z)}} \, (\det Y)^k \, dV(Z)
\vspace{-0.2cm}
\end{equation}
converges on $C_*:=\{ (s,w,u) \in \mathbb{C}^3/ \Re(s)>1, \Re(w)>1, \Re(s+w+u)<-1\}.$ We initially have that (\ref{v}) is bounded by
\vspace{-0.1cm}
\begin{equation*}
\sum_{T \in \mathcal{J}/SL_3(\mathbb{Z})} \abs[\big]{E(T|w,s,-s-w-u+2)} \int_{Z \in \mathcal{F}} \abs[\Big]{P_{k,T}(Z) \, \overline{f^*(Z)}} \, (\det Y)^k \, dV(Z).
\vspace{-0.05cm}
\end{equation*}
Since $f^* \in  \mathfrak{S}_{3,k}$ and $\mathcal{F}\subseteq V(\delta)$ for some $\delta>0$, there exist $C,C_1>0$ such that $ \abs[\big]{f^*(Z)}\leq C_1e^{-C \tr(Y)}$ for all $Z \in \mathcal{F}$ (see for instance Proposition~3 in \cite[p.~56]{Kli}). This fact, (\ref{detE}) and (\ref{poincareT}) imply that a sufficient condition for (\ref{v}) to converge is
\vspace{-0.05cm}
\begin{equation*}
    \sum_{T \in \mathcal{J}/SL_3(\mathbb{Z})} \abs[\big]{E(T|w,s,-s-w-u)} \text{ and } \int_{Z \in \mathcal{F}} (\det Y)^k \, e^{-C \tr(Y)} dV(Z)
\end{equation*}
both converge. The series converges for $(s,w,u) \in C_*$ by Lemma~\ref{serieE}. Regarding the integral, we show that it is bounded by a value of our gamma function $\Gamma_3$. Looking at the definition of $\mathcal{F}$ (see \cite[p.~29]{Kli}), writing $Z=X+iY$ as in (\ref{notation}) and recalling that $dV(Z)(\det Y)^4=dXdY=dx_1 dy_1 \cdot ... \cdot dx_6 dy_6$, one has
\begin{equation*}
\begin{aligned}
\hfilneg \int_{Z \in \mathcal{F}} (\det Y)^k \, e^{-C \tr(Y)} dV(Z)& \leq \int_{Y \in \mathcal{R}}\int_{{}^t\hspace{-0.02cm}X=X \in [0,1]^{3,3}} (\det Y)^{k-4} \, e^{-C \tr(Y)} dXdY \hspace{10000pt minus 1fil}\\
& \hspace{3.5cm} \leq \int_{Y \in \mathcal{P}_3} (\det Y)^{k-4} \, e^{-C \tr(Y)} dY.
\end{aligned}
\end{equation*}
Then, considering the bijection $Y\mapsto C^{-1}Y$ onto $\mathcal{P}_3,$ $d(C^{-1}Y)=C^{-6} \, dY$ and since $(\det Y)^2\,d\mu(Y)=dY$ and $(\det Y)^{k-2}\,e^{3(k-2)\pi i /2}=p_{0,0,k-2}(iY)$, the last integral is equal to $C^{6-3k} e^{3(2-k)\pi i /2}\,\Gamma_3(0,0,k-2)$. Therefore the series (\ref{v}) converges on $C_*$. 

The kernel is represented on $C_*\times\mathcal{H}_3 \subseteq C\times \mathcal{H}_3$ by (\ref{repreC}). Hence
\begin{equation}\label{v2}
\begin{aligned}
&\frac{\Gamma_3(-w,-s,s+w+u)}{(2\pi)^{s+2w+3u}} \langle {\Omega}_{k,s,w,u}, f^* \rangle\\
=&\int_{Z \in \mathcal{F}} \sum_{T \in \mathcal{J}/SL_3(\mathbb{Z})} \frac{1}{\varepsilon_T} E(T|w,s,-s-w-u+2) P_{k,T}(Z) \overline{f^*(Z)}(\det Y)^k dV(Z).
\end{aligned}
\end{equation}
Notice that we have used here that $\mathcal{J}/SL_3(\mathbb{Z}) = \mathcal{J}/GL_3(\mathbb{Z})$.
Since the series (\ref{v}) converges on $C_*$, we can use a theorem of \cite[p.~29]{Rud} (which is a corollary of Lebesgue's dominated convergence theorem) and conclude that the right-hand side of (\ref{v2}) is a convergent integral on $C_*$, where it is allowed to interchange integration and summation. Therefore, the left-hand side of (\ref{v2}) is well-defined on $C_*$ and is equal to
\vspace{-0.05cm}
\begin{equation*}
\begin{aligned}
&\displaystyle\sum_{T \in \mathcal{J}/SL_3(\mathbb{Z})} \dfrac{1}{\varepsilon_T} E(T|w,s,-s-w-u+2) \displaystyle\int_{Z \in \mathcal{F}} P_{k,T}(Z) \overline{f^*(Z)} (\det Y)^k dV(Z) \\[-0.3em]
=& \displaystyle\sum_{T \in \mathcal{J}/SL_3(\mathbb{Z})} \dfrac{1}{\varepsilon_T} E(T|w,s,-s-w-u+2) \langle P_{k,T}, f^* \rangle.
\end{aligned}
\vspace{-0.13cm}
\end{equation*}
Using the identity $\langle P_{k,T}, f^* \rangle=D_k (\det T)^{2-k} A_T$, where $D_k=\pi^{3/2} (4\pi)^{6-3k} \Gamma(k-2) \Gamma(k-5/2) \Gamma(k-3)$ (which is equivalent to (\ref{p,f*})) and (\ref{detE}), one has
\vspace{-0.1cm}
\begin{equation*}
\begin{array}{r@{}l}
 \dfrac{\Gamma_3(-w,-s,s+w+u)}{(2\pi)^{s+2w+3u}} & \langle {\Omega}_{k,s,w,u} , f^* \rangle \\[-0.2em]
&\;\;\;\;\;=D_k e^{(2-k)\pi i/2} \displaystyle\sum_{T \in\mathcal{J}/SL_3(\mathbb{Z})} \frac{A_T}{\varepsilon_T}E(T|w,s,-s-w-u+k)
\end{array}
\vspace{-0.05cm}
\end{equation*}
for any $(s,w,u)\in C_*$. By Lemma~\ref{Dirichletconverges} the last series is $D_f(w,s,-s-w-u+k).$
Hence, we have obtained the identity
\begin{equation*}
\frac{\Gamma_3(-w,-s,s+w+u)}{(2\pi)^{s+2w+3u}}\, \langle{\Omega}_{k,s,w,u}, f^* \rangle =D_k \;e^{(2-k)\pi i/2}D_f(w,s,-s-w-u+k)
\end{equation*}
whenever $(s,w,u) \in C^*.$ Changing $(s,w,u)$ to $(w,s,-s-w-u+k)$
and applying the first functional equation of Theorem~\ref{theoremanalyticcontinuation} one obtains
\begin{equation*}\label{laimportante}
\frac{\Gamma_3(-s,-w,-u+k)\,e^{(s+2w+3u)\pi i}}{(2\pi)^{-(s+2w+3u)}}\,  \langle {\Omega}_{k,s,w,u}, f^* \rangle =D_k \;e^{(2-k)\pi i/2}(2\pi)^{3k}D_f(s,w,u)
\end{equation*}
whenever $\Re(s)>1, \Re(w)>1$ and $\Re(u)>k+1.$ This identity together with (\ref{Gamma3property}) yield the formula in the theorem.
\end{proof}

\section{Application}\label{Application}
We begin this section with a remark, which will be useful in the proof of Proposition ~\ref{application}. In fact, the remark holds for Siegel cusp forms of any degree.

\vspace{0.2cm}
\begin{remark}\label{remark5.1}
Let $f \in  \mathfrak{S}_{3,k}$ and $\varepsilon >0$, then $f(Z)(\det Y)^{\varepsilon}$ is bounded on $\mathcal{F}$.
\end{remark}

\begin{proof} [Proof]
Let $C_1,C_2>0$ such that $|f(Z)| (\det Y)^{k/2} \leq C_1e^{-C_2(\det Y)^{1/3}}$ for all $Z \in \mathcal{F}$ (see for instance \cite[p.~57]{Kli}). Then, there exist $C_3,C_4,C_5>0$ such that
\vspace{-0.05cm}
\begin{equation*}
|f(Z)| (\det Y)^{\varepsilon} \leq C_1e^{-C_2(\det Y)^{1/3}}(\det Y)^{\varepsilon-k/2} \leq C_3e^{-C_4(\det Y)^{1/3}} \leq C_3e^{-C_5\delta}
\vspace{-0.05cm}
\end{equation*}
for all $Z \in \mathcal{F}$. The last inequality is due to (\ref{dety>delta}).
\end{proof}

\begin{proof} [Proof\nopunct] {\it of Proposition~\ref{application}.}
The main idea of the proof is to use the identity of Theorem~\ref{teorema2} to extend the twisted Koecher-Maass series. We consider
\begin{equation}\label{defKMseries}
\frac{\phi(s)\xi(2s)\phi(w)\xi(2w)\phi(s+w-1/2)\xi(2s+2w-1)}{\Gamma(-s-w-u+k)\Gamma(-w-u+k-1/2) \Gamma(-u+k-1)}D_f(s,w,u),
\vspace{-0.05cm}
\end{equation}
which is holomorphic on $I$ by Lemma~\ref{Dirichletconverges}, and define it on $\mathbb{C}^3$ as
\begin{equation}\label{contofKMseries}
C_k^{-1}(2\pi i)^{s+2w+3u} \langle \phi(s)\xi(2s)\phi(w)\xi(2w)\phi(s+w-1/2)\xi(2s+2w-1){\Omega}_{k,s,w,u}, f^* \rangle.
\end{equation}
where $C_k$ is defined in Theorem~\ref{teorema2}.

We prove that (\ref{contofKMseries}) is entire in the variable $s$. Let $s_0$ a complex number, $B(s_0,\varepsilon)$ the open ball centered at $s_0$ and radius $\varepsilon$ and define $\psi: B(s_0,\varepsilon) \times \mathcal{F} \to \mathbb{C}$,
\vspace{-0.05cm}
\begin{equation*}
    \psi(s,Z)=\phi(s)\xi(2s)\phi(w)\xi(2w)\phi(s+w-1/2)\xi(2s+2w-1){\Omega}_{k,s,w,u}(Z)\overline{f^*(Z)}(\det Y)^k.
\vspace{-0.05cm}
\end{equation*}
Since the set $\mathcal{F}$ is a measurable space, $\psi(s,\cdot)$ is a measurable function of $Z$ for each $s \in B(s_0,\varepsilon)$, $\psi(\cdot,Z)$ is holomorphic on $B(s_0,\varepsilon)$ for each $Z \in \mathcal{F}$ and $\psi$ is bounded on $B(s_0,\varepsilon) \times \mathcal{F}$ by Theorem~\ref{theoremanalyticcontinuation} and Remark~\ref{remark5.1}, we can use \cite[p.~229]{Rud} (which is a corollary of Cauchy's integral formula and Lebesgue’s dominated convergence theorem) and conclude that (\ref{contofKMseries}) is holomorphic on $B(s_0,\varepsilon)$. Finally (\ref{contofKMseries}) is entire in the variable $s$ since $s_0$ is an arbitrary complex number.

The same proof works for $w$ and $u$, then (\ref{defKMseries}) has a holomorphic continuation to $\mathbb{C}^3.$ 

We now turn to the functional equations. From (\ref{contofKMseries}) and the first functional equation of Theorem~\ref{theoremanalyticcontinuation} it follows that
\begin{equation*}
\begin{aligned}
&\dfrac{C_k e^{(s+2w+3u)\pi i}\phi(s)\xi(2s)\phi(w)\xi(2w)\phi(s+w-1/2)\xi(2s+2w-1)}{(2\pi i)^{s+2w+3u} \Gamma(-s-w-u+k)\Gamma(-w-u+k-1/2)\Gamma(-u+k-1)}D_f(s,w,u)\\[0.4em]
=&\, \langle \phi(s)\xi(2s)\phi(w)\xi(2w)\phi(s+w-1/2)\xi(2s+2w-1){\Omega}_{k,w,s,-s-w-u+k},f^* \rangle\\[0.4em]
=&\, \dfrac{C_k\phi(s)\xi(2s)\phi(w)\xi(2w)\phi(s+w-1/2)\xi(2s+2w-1)}{(2\pi i)^{-s-2w-3u+3k}\Gamma(s+w+u-1)\Gamma(w+u-1/2)\Gamma(u)}D_f(w,s,-s-w-u+k).
\end{aligned}
\end{equation*}
This identity easily gives the first functional equation claimed. The proof for the other functional equations is similar.
\end{proof}

\bmhead{Acknowledgments} The research was supported by Subvención a la Instalación en la Academia 2022 n°85220128 and Beca de postdoctorado en el extranjero Becas Chile n°74220047.
\vspace{-0.1cm}
\bibliography{sn-bibliography}

\end{document}